\documentclass[11pt]{amsart}

\usepackage{graphicx}%
\usepackage{multirow}%
\usepackage{amsmath,amssymb,amsfonts}%
\usepackage{amsthm}%
\usepackage{mathrsfs}%
\usepackage[title]{appendix}%
\usepackage{xcolor}%
\usepackage{textcomp}%
\usepackage{manyfoot}%
\usepackage{booktabs}%
\usepackage{algorithm}%
\usepackage{algorithmicx}%
\usepackage{algpseudocode}%
\usepackage{listings}%
\usepackage[colorlinks=true, linkcolor=red, urlcolor=blue, citecolor=blue]{hyperref}
\usepackage{cite}

\theoremstyle{definition}
\newtheorem{theorem}{Theorem}[section]

\newtheorem{lemma}[theorem]{Lemma}
\newtheorem{corollary}[theorem]{Corollary}
\theoremstyle{definition}
\newtheorem{definition}[theorem]{Definition}
\newtheorem{example}[theorem]{Example}
\theoremstyle{remark}
\newtheorem{remark}[theorem]{Remark}

\allowdisplaybreaks[1]
\usepackage[left=2.5cm, right=2.5cm, top=3cm, bottom=3cm]{geometry}

\title[ the algebraic Davis--Wielandt shell and norm-parallelism ]{On the geometry of the algebraic Davis--Wielandt shell and norm-parallelism in $C^*$-algebra}

\author[D. Bhattacharya, F. Kittaneh, A. Patra, S. Satpathy] { {Debarati Bhattacharya}$^{1}$, {Fuad Kittaneh}$^{2,3}$, {Arnab Patra}$^{4}$, {Sanchita Satpathy}$^{5}$}

\address{$^{[1]}$ Department of Mathematics, Indian Institute of Technology Bhilai, Chhattisgarh, India 491002}
        \email{\url{debaratib@iitbhilai.ac.in}}
        
\address{$^{[2]}$ Department of Mathematics, The University of Jordan, Amman, Jordan}
\email{\url{fkitt@ju.edu.jo}}

\address{$^{[3]}$ Department of Mathematics, Korea University, Seoul 02841, South Korea}
\email{\url{fkitt@ju.edu.jo}}
        
\address{$^{[4]}$ Department of Mathematics, Indian Institute of Technology Bhilai,  Chhattisgarh, India 491002.}
\email{\url{arnabp@iitbhilai.ac.in}}

\address{$^{[5]}$ Department of Mathematics, Indian Institute of Technology Bhilai,  Chhattisgarh, India 491002.}
\email{\url{sanchitasatpathy4@gmail.com}}


\begin{document}




\begin{abstract}
    This article is devoted to the study of the Davis--Wielandt shell and the Davis--Wielandt radii of elements in a $C^*$-algebra. Utilizing a state-space approach, several geometric properties of the algebraic Davis--Wielandt shell are established. Upper and lower bounds for the algebraic Davis--Wielandt radii are obtained including the Davis--Wielandt radius of the sum of $k$ elements. We also explore the relationship between norm-parallelism and the Davis--Wielandt radii of elements.

\end{abstract}

\maketitle

\section{Introduction}\label{sec1}

The study of numerical ranges and their corresponding radii forms a cornerstone of modern operator theory and functional analysis. For a bounded linear operator $T$ acting on a complex Hilbert space $\mathscr{H}$, the classical spatial numerical range $W(T)$ and the numerical radius $w(T)$ provide invaluable insight into the spectral, geometric, and structural properties of $T$.
	
	 Over the decades, this framework has been elegantly extended to an abstract setting. For a unital $C^*$-algebra $\mathfrak{A}$ with the unit $\boldsymbol{1}$, the \textit{algebraic numerical range} and the \textit{algebraic numerical radius} of an element $a \in \mathfrak{A}$, are defined respectively as: 
	\[ W_\mathfrak{A}(a)=\{f(a):f\in\mathfrak{S}(\mathfrak{A})\}, \]
	\[w_\mathfrak{A}(a)=\sup\{|z|:z\in W_\mathfrak{A}(a)\},
	\] 
    where $\mathfrak{S}(\mathfrak{A})$ denotes the state space of $\mathfrak{A}$, defined as the set of all positive linear functionals $f \in \mathfrak{A}^*$ such that $\|f\| = f(\boldsymbol{1}) = 1$. It is a standard result in operator algebra theory that $\mathfrak{S}(\mathfrak{A})$ is a weak$^*$-compact convex subset of the dual space $\mathfrak{A}^*$ \cite[Proposition~13.8]{zhu2018introduction}.

	The algebraic numerical range of $a\in\mathfrak{A}$ is a nonempty compact convex subset of $\mathbb{C}$ (see \cite{gau2021numerical}). Considering $\mathscr{B}(\mathscr{H})$, the algebra of all bounded linear operators in $\mathscr{H}$, we have $W_{\mathfrak{A}}(T)=\overline{W(T)}$. The intrinsic connection between the algebraic numerical radius and the $C^*$-norm is governed by the classical inequality (see \cite{duncan1971numerical}) 
    \begin{eqnarray}\label{*norm}
	\dfrac{1}{2}\|a\|\leq w_\mathfrak{A}(a)\leq\|a\|,
	\end{eqnarray} cementing its role as a fundamental structural metric within the algebra.

   To capture higher-order geometric and norm-related information simultaneously, Wielandt and Davis (e.g.\cite{davis1968shell,wielandt1955eigenvalues}) introduced the Davis--Wielandt shell. For an operator $T\in\mathscr{B}(\mathscr{H})$ the spatial Davis--Wielandt shell $DW(T)\subset\mathbb{C}\times\mathbb{R}$ and its associated radius $dw(T)$ are defined by: 
	\[
		DW(T)=\left\{\left(\langle Tx,x\rangle,\langle T^*Tx,x\rangle\right):x\in\mathscr{H},\|x\|=1\right\}
	,\] 
	\[
		dw(T)=\sup\left\{\sqrt{|\langle Tx,x\rangle|^2+\|Tx\|^4}:x\in\mathscr{H},\|x\|=1\right\}.
	\]  
$DW(T)$ is a nonempty bounded subset of $W(T)\times[0,\|T\|^2]$ and compact if $\mathscr{H}$ is finite dimensional. The Davis--Wielandt radius, unlike numerical radius, fails to be a norm on $\mathscr{B}(\mathscr{H})$. It is immediate that the numerical range $W(T)$ is precisely the projection of the Davis--Wielandt shell $DW(T)$ onto its first coordinate. Therefore, the Davis--Wielandt shell encodes additional information about the operator $T$ beyond what is contained in the numerical range alone. Indeed, in finite dimensions, the geometric structure of $DW(T)$ completely determines whether the operator is normal. More precisely, for $A\in \mathfrak{M}_n(\mathbb{C})$, the algebra of $n\times n$ complex matrices, $A$ is normal if and only if the set $DW(A)$, regarded as a subset of $\mathbb{C}\times\mathbb{R}\simeq\mathbb{R}^3$, forms a polyhedron. The geometric and spectral properties of $DW(T)$ have been explored in (e.g. \cite{chien2002davis,li2008davis}). Numerous works on estimations of $dw(T)$ were investigated in (e.g. \cite{alomari2023davis,bhunia2021new,bhunia2021bounds,bhunia2023further,zamani2020some,zamani2019norm}).
	
    Motivated by the utility of this joint spatial range, Aramba{\v{s}}i{\'c} et al. successfully migrated these concepts to the abstract setting of $C^*$-algebras. The \textit{algebraic Davis--Wielandt shell} $DW_{\mathfrak{A}}(a)$ and the \textit{algebraic Davis--Wielandt radius} $dw_{\mathfrak{A}}(a)$ of  $a\in\mathfrak{A}$ of an element $a\in\mathfrak{A}$ are defined respectively as:
	\[
		DW_{\mathfrak{A}}(a)=\left\{(f(a), f(a^*a)):f\in\mathfrak{S}(\mathfrak{A})\right\}
	,\]
	\[
		dw_{\mathfrak{A}}(a)=\sup\left\{\sqrt{|f(a)|^2+f(a^*a)^2}:f\in\mathfrak{S}(\mathfrak{A})\right\}.
	\] In \cite{arambavsic2018roberts} it is identified that algebraic Davis--Wielandt shell of $a\in\mathfrak{A}$ is compact and convex subspace of $\mathbb{C}\times\mathbb{R}$. 

 This transition from the spatial to the algebraic setting highlights a profound geometric distinction regarding \emph{convexity}. It is well-known that the spatial shell $DW(T)$ is convex whenever dim$\mathscr{H}\geq3$, but it can fail to be convex in lower dimensions. For instance, consider the basic nilpotent matrix in the algebra of $2\times2$ complex matrices $\mathfrak{M}_2(\mathbb{C})$:
 \[A=\begin{pmatrix}
0 & 1\\
0 & 0
\end{pmatrix}.\] 
A direct calculation shows that:
\[DW(A)=\left\{(\mu,r): r\in[0,1],\ |\mu|=\sqrt{r(1-r)}\right\}.\] 
Thus, if we identify \(\mathbb{C}\times\mathbb{R}\) with \(\mathbb{R}^3\), the spatial Davis--Wielandt shell is the sphere of radius \(1/2\) centered at
$\left(0,0,\dfrac{1}{2}\right)$ which is non-convex.

To resolve these low-dimensional geometric irregularities, we turn to the algebraic Davis--Wielandt shell, $DW_{\mathfrak{A}}(T)$, which provides a more topologically uniform framework for analysis. For an operator $T \in \mathfrak{A} \subseteq \mathscr{B}(\mathscr{H})$, the set of all states of the $C^*$-algebra $\mathfrak{A}$ is the weak$^*$-closed convex hull of its vector states $T \mapsto \langle Tx, x \rangle$ for unit vectors $x \in \mathscr{H}$. Since every vector state is also a state of $\mathfrak{A}$, it follows that the spatial shell is contained within the algebraic shell, giving the inclusion $\overline{DW(T)^\wedge} \subseteq DW_{\mathfrak{A}}(T)$. Conversely, since $DW_{\mathfrak{A}}(T)$ is both convex and compact, it follows $DW_{\mathfrak{A}}(T) \subseteq \overline{DW(T)^\wedge}$. Combining these conditions gives the identity $DW_{\mathfrak{A}}(T) = \overline{DW(T)^\wedge}$. 

Returning to the $2 \times 2$ nilpotent matrix under this algebraic lens, the shell successfully incorporates the convex hull
\begin{equation}
DW_{\mathfrak{A}}(A) = \overline{DW(A)^\wedge} = \left\{ (\mu, r) : r \in [0, 1], \, |\mu| \leq \sqrt{r(1-r)} \right\}.
\end{equation}
By capturing all interior points under the boundary curve, $DW_{\mathfrak{A}}(A)$ forms a solid convex ball. This contrast clearly illustrates that the algebraic shell provides a more topologically complete and regular framework for matrix and operator analysis than its spatial counterpart.

     The primary objective of this article is to systematically investigate geometric features of the algebraic Davis--Wielandt shell and unveil its deep structural connections to norm-parallelism. We also aim to establish several novel estimations for the algebraic Davis--Wielandt radius. The paper is organized as follows. In Section~\ref{sec2}, we study several geometric properties of the algebraic Davis--Wielandt shell and establish some basic results. Section~\ref{sec3} is devoted to deriving new upper and lower bounds for the algebraic Davis--Wielandt radius, extending the inequality frameworks to the sum of $k$ elements. In Section~\ref{sec4}, we explore the structural interplay between norm-parallelism and the Davis--Wielandt radii of elements. Finally, Section~\ref{sec5} presents concluding remarks.

\section{Geometric properties of the algebraic Davis--Wielandt shell}\label{sec2}

Before proceeding further, we present some lemmas that will be used throughout the paper.
\begin{lemma}\cite[Proposition~13.4]{zhu2018introduction}\label{cs inq}
	For every positive linear functional $f$ in a $C^*$-algebra $\mathfrak{A}$, we have
	\begin{eqnarray*}
		|f(a^*b)|^2\leq f(a^*a)f(b^*b),\qquad\ a, b\in\mathfrak{A}. 
	\end{eqnarray*}
	
\end{lemma}
\begin{lemma}\cite[Corollary~2.6]{mahapatra2024upper}\label{nw lm1}
    Let $\mathfrak{A}$ be a unital $C^*$-algebra. If a is a self-adjoint element in $\mathfrak{A}$ such that the spectrum of $a$ is contained in $[0,\infty)$ and $f\in\mathfrak{S}(\mathfrak{A})$, then for any $r\geq1$,
\[(f(a))^r\leq f(a^r).\]
\end{lemma}
	\noindent A representation of a $C^*$-algebra $\mathfrak{A}$ on a Hilbert space $\mathscr{H}$ is a $^*$-homomorphism $\pi:\mathfrak{A} \to \mathscr{B}(\mathscr{H})$, that is, $\pi$ satisfies $\pi(ab)=\pi(a)\pi(b)$ for all $a, b\in\mathfrak{A}$ and $\pi(a)^*=\pi(a^*)$ for all $a\in\mathfrak{A}$. In this regard, we state the famous Gelfand, Naimark, and Segal theorem.
	\begin{theorem}\label{gns}\cite[Theorem~1.6.3]{arveson1998invitation}
	Let $f\in\mathfrak{A}^*$ and $f\geq 0$. Then there is a representation $\pi$ of $\mathfrak{A}$ and a vector $\xi\in \mathscr{H}$ such that $f(a)=\langle\pi(a)\xi, \xi\rangle$ for every $a\in\mathfrak{A}$. Moreover, if $f\in \mathfrak{S}(\mathfrak{A})$, then $\|\xi\|=1$.
	\end{theorem}
We now mention some basic results for $DW_{\mathfrak{A}}(a)$ as follows.
	\begin{theorem}
	Let $a\in \mathfrak{A}$. Then the following statements hold:
		\begin{enumerate}
            \item [(i)] If $(\mu,r)\in DW_{\mathfrak{A}}(a)$, then $r\geq |\mu|^2$, that is, $DW_{\mathfrak{A}}(a)$ lies in a paraboloid.
			\item [(ii)] $DW_{\mathfrak{A}}(a)=DW_{\mathfrak{A}}(u^*au)$ for any unitary $u\in \mathfrak{A}$.
			\item [(iii)] $DW_{\mathfrak{A}}(\alpha a+\beta\boldsymbol{1})=\{(\alpha\mu+\beta, |\alpha|^2r+2\Re(\alpha\overline{\beta}\mu)+|\beta|^2):(\mu, r)\in DW_{\mathfrak{A}}(a)\}$ for any scalers $\alpha$ and $\beta$.
            \item[(iv)] If $a=a_1\oplus\cdots\oplus a_n$, where $a\in\mathfrak{A}=\oplus_{i=1}^n\mathfrak{A}_i$ for each $a_i\in\mathfrak{A}_i$, then \[DW_{\mathfrak{A}}(a)=\left(DW_{\mathfrak{A}_1}(a_1)\bigcup\cdots\bigcup DW_{\mathfrak{A}_n}(a_n)\right)^\wedge,\] where $\left(DW_{\mathfrak{A}_1}(a_1)\bigcup\cdots\bigcup DW_{\mathfrak{A}_n}(a_n)\right)^\wedge$ is the convex hull of $DW_{\mathfrak{A}_1}(a_1)\bigcup\cdots\bigcup DW_{\mathfrak{A}_n}(a_n)$.
		\end{enumerate}
	\end{theorem}
	\begin{proof}
        (i) If $(\mu,r)\in DW_{\mathfrak{A}}(a)$, then there exists $f\in\mathfrak{S}(\mathfrak{A})$ such that $f(a)=\mu$ and $f(a^*a)=r$. Then using Lemma~\ref{cs inq}, we get $|\mu|^2=|f(a)|^2\leq f(a^*a)=r$. Hence, 
        \[DW_{\mathfrak{A}}(a)\subseteq\{(\mu,r)\in\mathbb{C}\times[0,\infty):|\mu|^2\leq r\}.\]
	
		(ii) Let $f\in\mathfrak{S}(\mathfrak{A})$ and $u$ be unitary. Define $f_u(a)=f(uau^*)$. Then $f_u\in\mathfrak{S}(\mathfrak{A})$. Let $(f(a),f(a^*a))\in DW_{\mathfrak{A}}(a)$. Moreover, $f_u(u^*au)=f(a)$ and $f_u((u^*au)^*(u^*au))=f(a^*a)$ imply $(f(a),f(a^*a))\in DW_{\mathfrak{A}}(u^*au)$. Thus, $DW_{\mathfrak{A}}(a)\subseteq DW_{\mathfrak{A}}(u^*au)$. Conversely, defining $f_u'(a)=f(u^*au)$, one similarly obtains 
$DW_{\mathfrak{A}}(u^*au)\subseteq DW_{\mathfrak{A}}(a)$. 

 (iii) For any $a\in\mathfrak{A}$ and $f\in\mathfrak{S}(\mathfrak{A})$, we have \[
	DW_{\mathfrak{A}}(\alpha a+\beta\boldsymbol{1})=\left\{(f(\alpha a+\beta), f((\alpha a+\beta)^*(\alpha a+\beta))):f\in\mathfrak{S}(\mathfrak{A})\right\}.\]
Now, \[f(\alpha a+\beta)=\alpha f(a)+\beta=\alpha \mu+\beta\] and
 		 \[f((\alpha a+\beta)^*(\alpha a+\beta)) 
 	=|\alpha|^2f(a^*a)+2\Re(\alpha\overline{\beta}f(a))+|\beta|^2=|\alpha|^2r+2\Re(\alpha\overline{\beta}\mu)+|\beta|^2,\] where $\mu=f(a),\ r=f(a^*a)$.

 (iv) Let $a=a_1\oplus\cdots\oplus a_n\in \mathfrak{A}=\oplus_{i=1}^n \mathfrak{A}_i$. Every $f\in\mathfrak{S}(\mathfrak{A})$ can be written as $f=\sum_{i=1}^n \lambda_i f_i$, 
where $\lambda_i\ge 0$, $\sum_{i=1}^n \lambda_i=1$ and $f_i\in\mathfrak{S}(\mathfrak{A}_i)$ \cite[Lemma~8]{d2021pseudo}. 
Hence, $(f(a),f(a^*a))=\sum_{i=1}^n \lambda_i (f_i(a_i),f_i(a_i^*a_i))
\in \left(DW_{\mathfrak{A}_1}(a_1)\bigcup\cdots\bigcup DW_{\mathfrak{A}_n}(a_n)\right)^\wedge$,
which shows
\[
DW_{\mathfrak{A}}(a)\subseteq \left(DW_{\mathfrak{A}_1}(a_1)\bigcup\cdots\bigcup DW_{\mathfrak{A}_n}(a_n)\right)^\wedge.
\]
Conversely, let $(u_i,v_i)\in DW_{\mathfrak{A}_i}(a_i)$ and $\lambda_i\ge 0$ with $\sum_{i=1}^n \lambda_i=1$. 
Choose $f_i\in\mathfrak{S}(\mathfrak{A}_i)$ such that $(u_i,v_i)=(f_i(a_i),f_i(a_i^*a_i))$. 
Then $f=\sum_{i=1}^n \lambda_i f_i\in\mathfrak{S}(\mathfrak{A})$ and
$(f(a),f(a^*a))=\sum_{i=1}^n \lambda_i (u_i,v_i)\in DW_{\mathfrak{A}}(a)$.
Thus,
\[
\left(DW_{\mathfrak{A}_1}(a_1)\bigcup\cdots\bigcup DW_{\mathfrak{A}_n}(a_n)\right)^\wedge \subseteq DW_{\mathfrak{A}}(a).
\]
	\end{proof}
   To further study the geometrical properties of the algebraic Davis--Wielandt shell we need the following notations. For $a\in\mathfrak{A}$ and $\mu\in W_{\mathfrak{A}}(a)$ consider
    \[\mathscr{L}_{\mu}^{\mathfrak{A}}(a)=\left\{r:(\mu,r)\in DW_{\mathfrak{A}}(a)\right\}\subseteq [0,\|a\|^2].\] It is easy to verify that $\mathscr{L}_{\mu}^{\mathfrak{A}}(a)$ is a compact subset of $\mathbb{R}$ (see \cite{arambavsic2018roberts}). The height function $h_a^{\mathfrak{A}}:W_{\mathfrak{A}}(a)\to\mathbb{R}$ defined as \begin{eqnarray}\label{h-new}
        h_a^{\mathfrak{A}}(\mu)=\sup\mathscr{L}_{\mu}^{\mathfrak{A}}(a). 
     \end{eqnarray}
      The upper boundary of $DW_{\mathfrak{A}}(a)$ (see \cite{arambavsic2018roberts}) is the set
    \[
    \partial_{+}DW_{\mathfrak{A}}(a)=\left\{(\mu,\ h_a^{\mathfrak{A}}(\mu))\in\mathbb{C}\times\mathbb{R}:\mu\in W_{\mathfrak{A}}(a),h_a^{\mathfrak{A}}(\mu)=\sup\mathscr{L}_{\mu}^{\mathfrak{A}}(a)\right\}.
     \] In other words, $\partial_{+}DW_{\mathfrak{A}}(a)$ is precisely the graph of the function $h_a^{\mathfrak{A}}$. Similarly, we can define the lower boundary of $DW_{\mathfrak{A}}(a)$ is the set
     \[\partial_{-}DW_{\mathfrak{A}}(a)=\left\{(\mu,g_a^{\mathfrak{A}}(\mu))\in\mathbb{C}\times\mathbb{R}:\mu\in W_{\mathfrak{A}}(a),\ g_a^{\mathfrak{A}}(\mu)=\inf\mathscr{L}_{\mu}^{\mathfrak{A}}(a)\right\}.\]
 In the following theorem, we prove the convexity of $\mathscr{L}^{\mathfrak{A}}_\mu(a)$.
    \begin{theorem}
         $\mathscr{L}_{\mu}^{\mathfrak{A}}(a)$ is a convex set in $\mathbb{R}$ for each $\mu\in W_{\mathfrak{A}}(a)$.
    \end{theorem}
    \begin{proof}
        Let $r_1,r_2\in\mathscr{L}_{\mu}^{\mathfrak{A}}(a)$. Then there exist $f_i\in\mathfrak{S}(\mathfrak{A})$ such that $f_i(a)=\mu$, $f_i(a^*a)=r_i$, $i=1,2$. For any $\lambda\in[0,1]$, define $f=\lambda f_1+(1-\lambda) f_2$. Since $\mathfrak{S}(\mathfrak{A})$ is convex, $f\in\mathfrak{S}(\mathfrak{A})$. Now, 
        \[f(a)=\mu\ \mathrm{and}\ f(a^*a)=\lambda r_1+(1-\lambda)r_2 \implies (\mu,\lambda r_1+(1-\lambda)r_2)\in DW_{\mathfrak{A}}(a).\] Hence, $\lambda r_1+(1-\lambda)r_2\in\mathscr{L}_{\mu}^{\mathfrak{A}}(a)$. This proves the result.
    \end{proof}
\begin{remark}
    As $\mathscr{L}_{\mu}^{\mathfrak{A}}(a)$ is a compact convex set in $\mathbb{R}$, it must be a closed interval.
\end{remark}
We distinguish between the spatial Davis--Wielandt shell and algebraic Davis--Wielandt shell. We first recall that for any $T \in \mathscr{B}(\mathscr{H})$ and $\mu \in W(T)$ (see \cite{li2008davis}),
\[
\mathscr{L}_\mu(T)= \{r:(\mu,r)\in DW(T)\},
\] where $DW(T)$ denotes the spatial Davis--Wielandt shell. 
\begin{lemma}
    If dim$\mathscr{H}<\infty$ then $\mathscr{L}_\mu(T)$ is compact in $\mathbb{R}$ for each $\mu\in W(T)$.
\end{lemma}
\begin{proof}
  Let $r_n\in\mathscr{L}_\mu(T)$ be such that $r_n\to r$. Then for each $n$ and $\mu\in W(T)$, $(\mu,r_n)\in DW(T)$. As $DW(T)$ is compact if $\mathscr{H}$ is finite dimensional, we have $(\mu,r)\in DW(T)$, this implies $r\in\mathscr{L}_\mu(T)$. 
\end{proof}
\begin{lemma}
    If dim$\mathscr{H}\geq3$, then $\mathscr{L}_\mu(T)$ is a convex set in $\mathbb{R}$ for each $\mu\in W(T)$.
\end{lemma}
\begin{proof}
    We know $DW(T)$ is convex if dim$\mathscr{H}\geq3$. Then for $\mu\in W(T)$, $r_1,r_2\in\mathscr{L}_\mu(T)$ and $\lambda\in[0,1]$, $\lambda(\mu,r_1)+(1-\lambda)(\mu,r_2)\in DW(T)$, which implies $\lambda r_1+(1-\lambda)r_2\in\mathscr{L}_\mu(T)$.
\end{proof}
In particular, if $\mathfrak{A}=\mathscr{B}(\mathscr{H})$ the following lemma gives a relation between $\mathscr{L}_\mu(T)$ and $\mathscr{L}_\mu^{\mathfrak{A}}(T)$ for $T\in\mathscr{B}(\mathscr{H})$.
     \begin{lemma}
         If $T\in\mathfrak{A}\subseteq\mathscr{B}(\mathscr{H})$, then $\mathscr{L}_{\mu}^{\mathfrak{A}}(T)=\overline{\mathscr{L}_{\mu}(T)^\wedge}$ for each $\mu\in W(T)$.
     \end{lemma} 
     \begin{proof}
         From $DW(T)\subseteq DW_{\mathfrak{A}}(T)$, it is easy to see $\mathscr{L}_\mu(T)\subseteq\mathscr{L}_\mu^{\mathfrak{A}}(T)$. This implies $\overline{\mathscr{L}_\mu(T)^\wedge}\subseteq\mathscr{L}_\mu^{\mathfrak{A}}(T)$ for each $\mu\in W(T)$. On the other hand, if $r\in\mathscr{L}_\mu^{\mathfrak{A}}(T)$ for each $\mu\in W(T)$, then $(\mu,r)\in DW_{\mathfrak{A}}(T)=\overline{DW(T)^\wedge}$. Thus, $r=f(T^*T)$ for some $f\in\mathfrak{S}(\mathfrak{A})$. Since the set of all states of a unital $C^*$-algebra $\mathfrak{A}\subseteq\mathscr{B}(\mathscr{H})$ is a weak$^*$-closed convex hull of the set of all vector states of $\mathfrak{A}$, it follows that $r\in\mathscr{L}_{\mu}(T)^\wedge\subseteq\overline{\mathscr{L}_{\mu}(T)^\wedge}$. Thus, $\mathscr{L}_\mu^{\mathfrak{A}}(T)\subseteq\overline{\mathscr{L}_{\mu}(T)^\wedge}$.
     \end{proof}
   In addition, for $\mu\in W(T)$
\[
\mathscr{L}_{\mu}^{\mathfrak{A}}(T)=
\begin{cases}
\overline{\mathscr{L}_{\mu}(T)}, & \text{if }\dim\mathscr{H}\geq 3,\\
\mathscr{L}_{\mu}(T), & \text{if }3\leq \dim\mathscr{H}<\infty.
\end{cases}
\]
     
     In this context, considering $\mathfrak{A}=\mathfrak{M}_2(\mathbb{C})$, the algebra of $2\times2$ complex matrices, we can observe the characterization of $\mathscr{L}_{\mu}(T)$ and $\mathscr{L}_{\mu}^{\mathfrak{A}}(T)$ in the following example. 
    \begin{example}\label{eg1}
         Take $A=\begin{pmatrix}
0 & 1\\
0 & 0
\end{pmatrix}\in\mathfrak{M}_2(\mathbb{C})$. Let $x=(x_1,x_2)^t\in\mathbb{C}^2$. Then for any $(\mu,r)\in DW(A)$, we have 
\[\mu=\langle Ax,x\rangle=x_2\overline{x_1}\in W(A)=\left\{z\in\mathbb{C}:|z|\leq\dfrac{1}{2}\right\}\quad\mathrm{and}\quad r=\|Ax\|^2=|x_2|^2\in[0,1].\] Also, from $|x_1|^2+|x_2|^2=1$, we obtain
$|\mu|^2= r(1-r),\ r\in[0,1]$. Thus, \[\mathscr{L}_{\mu}(A)=\{r\in[0,1]:r(1-r)=|\mu|^2\},\] equivalently,\[\mathscr{L}_{\mu}(A)=\left\{\dfrac{1-\sqrt{1-4|\mu|^2}}{2},\dfrac{1+\sqrt{1-4|\mu|^2}}{2}\right\},\] which is not convex for $|\mu|<\dfrac{1}{2}$. For $\mu=\pm\dfrac{1}{2}$, $\mathscr{L}_{\mu}(A)=\left\{\dfrac{1}{2}\right\}$. Also, 
\[\mathscr{L}_{\mu}^{\mathfrak{A}}(A)=\overline{\mathscr{L}_{\mu}(A)^\wedge}=\left[\dfrac{1-\sqrt{1-4|\mu|^2}}{2},\dfrac{1+\sqrt{1-4|\mu|^2}}{2}\right],\] which is an interval in $\mathbb{R}$ for $|\mu|\leq\dfrac{1}{2}$ .
     \end{example}
     It is easy to verify that if $a$ is a normal element in $\mathfrak{A}$ then $\mathscr{L}_{\mu}^{\mathfrak{A}}(a)=\mathscr{L}_{\overline{\mu}}^{\mathfrak{A}}(a^*)$ for each $\mu\in W_{\mathfrak{A}}(a)$. However, the converse is not true, that is, both sets $\mathscr{L}_{\mu}^{\mathfrak{A}}(a)$ and $\mathscr{L}_{\overline{\mu}}^{\mathfrak{A}}(a^*)$ may be equal for each $\mu\in W_{\mathfrak{A}}(a)$ although $a$ is not normal. To clarify this, we consider the following example.
     \begin{example}
         Take $\mathfrak{A}=\mathfrak{M}_3(\mathbb{C})$, the algebra of $3\times3$ complex matrices. Let $A=\begin{pmatrix}
0 & 0 & 1\\
0 & 0 & 0\\
0 & 0 & 0
\end{pmatrix}\in\mathfrak{M}_3(\mathbb{C})$. Then $A^*A=\begin{pmatrix}
0 & 0 & 0\\
0 & 0 & 0\\
0 & 0 & 1
\end{pmatrix}$ and $AA^*=\begin{pmatrix}
1 & 0 & 0\\
0 & 0 & 0\\
0 & 0 & 0
\end{pmatrix}$ imply that $A$ is not normal. Let $x=(x_1,x_2,x_3)^t\in\mathbb{C}^3$. Then for any $(\mu,r)\in DW(A)$, we have 
\[\mu=\langle Ax,x\rangle=x_3\overline{x_1}\in W(A)=\left\{z\in\mathbb{C}:|z|\leq\dfrac{1}{2}\right\}\quad\mathrm{and}\quad r=\|Ax\|^2=|x_3|^2\in[0,1].\] Also, from $|x_1|^2=1-|x_2|^2-|x_3|^2\leq 1-|x_3|^2$ (since $|x_2|^2\geq0$), we obtain
$|\mu|^2\leq r(1-r),\ r\in[0,1]$. Thus, \[\mathscr{L}_{\mu}(A)=\{r\in[0,1]:r(1-r)\geq|\mu|^2\},\] or equivalently, \[\mathscr{L}_{\mu}(A)=\left[\dfrac{1-\sqrt{1-4|\mu|^2}}{2},\dfrac{1+\sqrt{1-4|\mu|^2}}{2}\right]\quad\mathrm{for}\ |\mu|\leq\dfrac{1}{2}.\] Hence, 
\[\mathscr{L}_{\mu}^{\mathfrak{A}}(A)=\mathscr{L}_{\mu}(A)=\left[\dfrac{1-\sqrt{1-4|\mu|^2}}{2},\dfrac{1+\sqrt{1-4|\mu|^2}}{2}\right]\quad\mathrm{for}\ |\mu|\leq\dfrac{1}{2}.\]
By a similar argument, one can verify that
\[\mathscr{L}_{\overline{\mu}}^{\mathfrak{A}}(A^*)=\left[\dfrac{1-\sqrt{1-4|\mu|^2}}{2},\dfrac{1+\sqrt{1-4|\mu|^2}}{2}\right]\quad\mathrm{for}\ |\mu|\leq\dfrac{1}{2}.\]
     \end{example}
    The next theorem gives an equivalent condition for the inclusion relations of algebraic Davis--Wielandt shells of two elements.
     \begin{theorem}
         Let $\mathfrak{A}_1,\mathfrak{A}_2$ be two $C^*$-algebras, and also let $a\in\mathfrak{A}_1,b\in\mathfrak{A}_2$ be two elements. Then the following conditions are equivalent:
         \begin{enumerate}
         \item [(i)] $DW_{\mathfrak{A}_1}(a)\subseteq DW_{\mathfrak{A}_2}(b)$.
         \item [(ii)] $\mathscr{L}_{\mu}^{\mathfrak{A}_1}(a)\subseteq \mathscr{L}_{\mu}^{\mathfrak{A}_2}(b)$ for each $\mu\in W_{\mathfrak{A}_1}(a)$.
         \end{enumerate}
     \end{theorem}
     \begin{proof}
         To prove (i)$\implies$ (ii), first assume that $DW_{\mathfrak{A}_1}(a)\subseteq DW_{\mathfrak{A}_2}(b)$. Also, let $\mu\in W_{\mathfrak{A}_1}(a)$ and $r\in\mathscr{L}_{\mu}^{\mathfrak{A}_1}(a)$. 
         Then, by definition, $(\mu,r)\in DW_{\mathfrak{A}_1}(a)\subseteq DW_{\mathfrak{A}_2}(b)$. This implies $r\in\mathscr{L}_{\mu}^{\mathfrak{A}_2}(b)$. Hence, $\mathscr{L}_{\mu}^{\mathfrak{A}_1}(a)\subseteq \mathscr{L}_{\mu}^{\mathfrak{A}_2}(b)$ for each $\mu\in W_{\mathfrak{A}_1}(a)$.

          \noindent (ii) $\implies$ (i): Let $\mathscr{L}_{\mu}^{\mathfrak{A}_1}(a)\subseteq \mathscr{L}_{\mu}^{\mathfrak{A}_2}(b)$ for each $\mu\in W_{\mathfrak{A}_1}(a)$. This implies $W_{\mathfrak{A}_1}(a)\subseteq W_{\mathfrak{A}_2}(b)$. Suppose $(\mu,r)\in DW_{\mathfrak{A}_1}(a)$. Then $\mu\in W_{\mathfrak{A}_1}(a)\subseteq W_{\mathfrak{A}_2}(b)$ and $r\in\mathscr{L}_{\mu}^{\mathfrak{A}_1}(a)$. By assumption, $r\in \mathscr{L}_{\mu}^{\mathfrak{A}_2}(b)$ implies $(\mu,r)\in DW_{\mathfrak{A}_2}(b)$. Therefore, $DW_{\mathfrak{A}_1}(a)\subseteq DW_{\mathfrak{A}_2}(b)$.
     \end{proof}
      \begin{theorem}\label{supL}
         Let $a\in\mathfrak{A}$. Then for each $\mu\in W_{\mathfrak{A}}(a)$, we have $h_a^{\mathfrak{A}}(\mu)=h_{a^*}^{\mathfrak{A}}(\overline{\mu})$.
     \end{theorem}
     \begin{proof}
         For $a\in\mathfrak{A}$ and $\mu\in W_{\mathfrak{A}}(a)$, we have
         \[\mathscr{L}_{\mu}^{\mathfrak{A}}(a)=\left\{f(a^*a):f\in\mathfrak{S}(\mathfrak{A}),f(a)=\mu\right\}\] and 
         \[\mathscr{L}_{\overline{\mu}}^{\mathfrak{A}}(a^*)=\left\{f(aa^*):f\in\mathfrak{S}(\mathfrak{A}),f(a^*)=\overline{\mu}\right\}.\] Since $f(a^*)=\overline{f(a)}$ holds for all $f\in\mathfrak{S}(\mathfrak{A})$, to prove the required result, it suffices to show that \[\sup\{f(a^*a):f\in\mathfrak{S}(\mathfrak{A})\}=\sup\{f(aa^*):f\in\mathfrak{S}(\mathfrak{A})\}.\] 
         Now, for positive elements, $\|a^*a\|=\sup\{f(a^*a):f\in\mathfrak{S}(\mathfrak{A})\}$ and $\|aa^*\|=\sup\{f(aa^*):f\in\mathfrak{S}(\mathfrak{A})\}$. Also, we have $\|a^*a\|=\|aa^*\|=\|a\|^2.$ This completes the proof.
     \end{proof} 
      Next, we want to study some characteristics of the height function $h_a^{\mathfrak{A}}$.
    \begin{theorem}
Let $a\in\mathfrak{A}$, $\mu\in W_{\mathfrak{A}}(a)$ and the height function 
$h_a^{\mathfrak{A}}$ be defined in \eqref{h-new}. Then on $W_{\mathfrak{A}}(a)$
\begin{enumerate}
    \item [(i)] $h_a^{\mathfrak{A}}$ is non-negative.
    \item [(ii)] $h_a^{\mathfrak{A}}$ is concave.
    \item [(iii)] $h_a^{\mathfrak{A}}$ is upper semi-continuous.
\end{enumerate}
\end{theorem}
     \begin{proof}
       (i) Since $\mathscr{L}_{\mu}^{\mathfrak{A}}(a)\subseteq [0,\|a\|^2]$ for $a\in\mathfrak{A}$ and each $\mu\in W_{\mathfrak{A}}(a)$, it follows that $h_a^{\mathfrak{A}}(\mu)\geq0$.

      \noindent (ii) Let $\mu_1,\mu_2\in W_{\mathfrak{A}}(a)$ and $0\leq\lambda\leq1$. Take $r_1\in\mathscr{L}_{\mu_1}^{\mathfrak{A}}(a)$ and $r_2\in\mathscr{L}_{\mu_2}^{\mathfrak{A}}(a)$. Then $(\mu_1,r_1),(\mu_2,r_2)\in DW_{\mathfrak{A}}(a)$. Since $DW_{\mathfrak{A}}(a)$ is convex,
       \[\lambda(\mu_1,r_1)+(1-\lambda)(\mu_2,r_2)=(\lambda\mu_1+(1-\lambda)\mu_2,\lambda r_1+(1-\lambda)r_2)\in DW_{\mathfrak{A}}(a).\] Hence, $\lambda r_1+(1-\lambda)r_2\in\mathscr{L}_{\lambda\mu_1+(1-\lambda)\mu_2}^{\mathfrak{A}}(a)$. Therefore,
       \begin{eqnarray*}
         && \lambda r_1+(1-\lambda)r_2\leq h_a^{\mathfrak{A}}(\lambda\mu_1+(1-\lambda)\mu_2)\\
          &\implies&\lambda\sup_{r_1\in\mathscr{L}_{\mu_1}^{\mathfrak{A}}(a)}\mathscr{L}_{\mu_1}^{\mathfrak{A}}(a)+(1-\lambda)\sup_{r_2\in\mathscr{L}_{\mu_2}^{\mathfrak{A}}(a)}\mathscr{L}_{\mu_2}^{\mathfrak{A}}(a)\leq h_a^{\mathfrak{A}}(\lambda\mu_1+(1-\lambda)\mu_2)\\
          &\implies&\lambda h_a^{\mathfrak{A}}(\mu_1)+(1-\lambda) h_a^{\mathfrak{A}}(\mu_2)\leq h_a^{\mathfrak{A}}(\lambda\mu_1+(1-\lambda)\mu_2).
       \end{eqnarray*}

      \noindent (iii) To show that $h_a^{\mathfrak{A}}$ is upper semi-continuous, it is enough to prove that $\mathrm{hyp}(h_a^{\mathfrak{A}})$ is closed, where \[\mathrm{hyp}(h_a^{\mathfrak{A}})=\{(\mu,t)\in W_{\mathfrak{A}}(a)\times\mathbb{R}:t\leq h_a^{\mathfrak{A}}(\mu)\}\] denotes the  hypograph of $h_a^{\mathfrak{A}}$. Let $(\mu_n,t_n)\in\mathrm{hyp}(h_a^{\mathfrak{A}})$ be such that $(\mu_n,t_n)\to(\mu,t)$ in $W_{\mathfrak{A}}(a)\times\mathbb{R}$. We want to show that $t\leq h_a^{\mathfrak{A}}(\mu)$. Since $(\mu_n,t_n)\in\mathrm{hyp}(h_a^{\mathfrak{A}})$, by definition, $t_n\leq h_a^{\mathfrak{A}}(\mu_n)=\sup\mathscr{L}_{\mu_n}^{\mathfrak{A}}(a)$. Thus, for each $n$, we can choose $r_n\in\mathscr{L}_{\mu_n}^{\mathfrak{A}}(a)$ such that $r_n\geq t_n-\dfrac{1}{n}$. As, $r_n\in\mathscr{L}_{\mu_n}^{\mathfrak{A}}(a)$ for each $\mu_n\in W_{\mathfrak{A}}(a)$, we have $(\mu_n,r_n)\in DW_{\mathfrak{A}}(a)$. Since $DW_{\mathfrak{A}}(a)$ is compact, there exists a subsequence $(\mu_{n_k},r_{n_k})$ such that $(\mu_{n_k},r_{n_k})\to(\mu,r)$ with $(\mu,r)\in DW_{\mathfrak{A}}(a)$. Hence, $r\in\mathscr{L}_{\mu}^{\mathfrak{A}}(a)$. Also, from $r_{n_k}\geq t_{n_k}-\dfrac{1}{n_k}$, we obtain $r\geq t$. Hence, $h_a^{\mathfrak{A}}(\mu)\geq r \geq t$, which implies $t \leq h_a^{\mathfrak{A}}(\mu)$. Therefore, $(\mu,t)\in \mathrm{hyp}(h_a^{\mathfrak{A}})$, and consequently $\mathrm{hyp}(h_a^{\mathfrak{A}})$ is closed. This completes the proof.
     \end{proof}
     We now examine the relationship between $DW_{\mathfrak{A}}(a)$ and $DW_{\mathfrak{A}}(a^*)$. Although one may expect $DW_{\mathfrak{A}}(a^*)=\left\{(\overline{\mu},r):(\mu,r)\in DW_{\mathfrak{A}}(a)\right\}$, this is true for a normal element $a$ and fails in general, as the following example illustrates.
     \begin{example}
         Let $\mathfrak{A}=\mathscr{B}(\mathscr{H})$, be the algebra of all bounded linear operators in an infinite-dimensional Hilbert space $\mathscr{H}$. Also, consider the unilateral shift operator $S\in\mathscr{B}(\mathscr{H})$. It is observed in \cite[Example~3.5]{li2008davis} that 
         \[DW(S)=\left\{(\mu,1):\mu\in\mathbb{C},|\mu|<1\right\}\] and 
          \[DW(S^*)=\left\{(\mu,1):\mu\in\mathbb{C},|\mu|<1\right\}\bigcup\left\{(\mu,r):\mu\in\mathbb{C},|\mu|^2\leq r<1\right\}.\] Hence,\[DW_{\mathfrak{A}}(S)=\overline{DW(S)}=\left\{(\mu,1):\mu\in\mathbb{C},|\mu|\leq1\right\}\] and
\[DW_{\mathfrak{A}}(S^*)=\overline{DW(S^*)}=\left\{(\mu,r):\mu\in\mathbb{C},|\mu|^2\leq r\leq 1\right\}.\] 
          This implies $(0,0)\in DW_{\mathfrak{A}}(S^*)$, but $(0,0)\notin \left\{(\overline{\mu},r):(\mu,r)\in DW_{\mathfrak{A}}(S)\right\}$.
     \end{example}
     To compare $DW_{\mathfrak{A}}(a)$ and  $DW_{\mathfrak{A}}(a^*)$, it is enough to compare $\mathscr{L}_{\mu}^{\mathfrak{A}}(a)$ and $\mathscr{L}_{\overline{\mu}}^{\mathfrak{A}}(a^*)$. In this regard, we have the following result.
     \begin{theorem}\label{conv nt}
         Let $a\in\mathfrak{A}$. Then the following statements are equivalent:
         \begin{enumerate}
         \item [(i)] $DW_{\mathfrak{A}}(a^*)\subseteq \left\{(\overline{\mu},r):(\mu,r)\in DW_{\mathfrak{A}}(a)\right\}$.
         \item [(ii)] $\mathscr{L}_{\overline{\mu}}^{\mathfrak{A}}(a^*)\subseteq\mathscr{L}_{\mu}^{\mathfrak{A}}(a)$ for every $\mu\in W_{\mathfrak{A}}(a)$.
         \end{enumerate}
         As a result,
         \[DW_{\mathfrak{A}}(a^*)=\left\{(\overline{\mu},r):(\mu,r)\in DW_{\mathfrak{A}}(a)\right\}\] if and only if for every $f\in\mathfrak{S}(\mathfrak{A})$, there exist $\phi,\psi\in\mathfrak{S}(\mathfrak{A})$ such that the following conditions are satisfied:
         \[f(a)=\phi(a)\quad\mathrm{with}\quad\phi(a^*a)\leq f(aa^*)\quad \mathrm{and}\quad
f(a)=\psi(a)\quad\mathrm{with}\quad\psi(aa^*)\leq f(a^*a).\]
     \end{theorem}
     \begin{proof}
         (i)$\iff$(ii): Note that
         \[DW_{\mathfrak{A}}(a)=\bigcup_{\mu\in W_{\mathfrak{A}}(a)}\left\{(\mu,r):r\in\mathscr{L}_{\mu}^{\mathfrak{A}}(a)\right\}\] and
         \[DW_{\mathfrak{A}}(a^*)=\bigcup_{\mu\in W_{\mathfrak{A}}(a)}\left\{(\overline{\mu},r):r\in\mathscr{L}_{\overline{\mu}}^{\mathfrak{A}}(a^*)\right\}.\] This follows that $DW_{\mathfrak{A}}(a^*)\subseteq \left\{(\overline{\mu},r):(\mu,r)\in DW_{\mathfrak{A}}(a)\right\}$ if and only if $\mathscr{L}_{\overline{\mu}}^{\mathfrak{A}}(a^*)\subseteq\mathscr{L}_{\mu}^{\mathfrak{A}}(a)$ for every $\mu\in W_{\mathfrak{A}}(a)$.

         Suppose $DW_{\mathfrak{A}}(a^*)=\left\{(\overline{\mu},r):(\mu,r)\in DW_{\mathfrak{A}}(a)\right\}$, or equivalently $DW_{\mathfrak{A}}(a)=\left\{(\mu,r):(\overline{\mu},r)\in DW_{\mathfrak{A}}(a^*)\right\}$. Then $f(a^*)=\overline{\mu}$ and $f(aa^*)=r$ for some $f\in\mathfrak{S}(\mathfrak{A})$. Then $(\overline{\mu},r)\in DW_{\mathfrak{A}}(a^*)$ implies $(\mu,r)\in DW_{\mathfrak{A}}(a)$, where $\mu=f(a)$. Hence, $r\in\mathscr{L}_{\mu}^{\mathfrak{A}}(a)$. Thus, there exists $\phi\in\mathfrak{S}(\mathfrak{A})$ such that $\phi(a)=f(a)$ and $\phi(a^*a)=r=f(aa^*)$. On the other hand, if $r'=f(a^*a)$, then $(\mu,r')\in DW_{\mathfrak{A}}(a)$, hence, $(\overline{\mu},r')\in DW_{\mathfrak{A}}(a^*)$ implies $r'\in\mathscr{L}_{\overline{\mu}}^{\mathfrak{A}}(a^*)$. Then there exists $\psi\in\mathfrak{S}(\mathfrak{A})$ such that $f(a)=\psi(a)$ and $\psi(aa^*)=r'=f(a^*a)$.

          Conversely, for any $f\in\mathfrak{S}(\mathfrak{A})$, let $\phi,\psi\in\mathfrak{S}(\mathfrak{A})$ be such that the above mentioned conditions are satisfied. Let $\mu\in W_{\mathfrak{A}}(a)$. Take any $r\in\mathscr{L}_{\overline{\mu}}^{\mathfrak{A}}(a^*)$. Then we have $\mu=f(a)$ and $f(aa^*)=r$ for some $f\in\mathfrak{S}(\mathfrak{A})$. By assumption, there exist $\phi\in\mathfrak{S}(\mathfrak{A})$ such that $\phi(a)=f(a)=\mu$ and $\phi(a^*a)\leq r=f(aa^*)$. Let $s=\phi(a^*a)$. Then $s\in\mathscr{L}_{\mu}^{\mathfrak{A}}(a)$ and $s\leq r$. Hence, for every $r\in\mathscr{L}_{\overline{\mu}}^{\mathfrak{A}}(a^*)$, there exists $s\in\mathscr{L}_{\mu}^{\mathfrak{A}}(a)$ such that $s\leq r$. Thus,
         \[
             \inf\mathscr{L}_{\mu}^{\mathfrak{A}}(a)\leq r\quad \mathrm{for}\ \mathrm{all}\ r\in\mathscr{L}_{\overline{\mu}}^{\mathfrak{A}}(a^*)
             \implies \inf\mathscr{L}_{\mu}^{\mathfrak{A}}(a)\leq\inf\mathscr{L}_{\overline{\mu}}^{\mathfrak{A}}(a^*).
         \] Similarly, considering the another condition, we can get $\inf\mathscr{L}_{\overline{\mu}}^{\mathfrak{A}}(a^*)\leq\inf\mathscr{L}_{\mu}^{\mathfrak{A}}(a)$. Thus, $\inf\mathscr{L}_{\mu}^{\mathfrak{A}}(a)=\inf\mathscr{L}_{\overline{\mu}}^{\mathfrak{A}}(a^*)$. Also, from Theorem~\ref{supL} we get $\sup\mathscr{L}_{\mu}^{\mathfrak{A}}(a)=\sup\mathscr{L}_{\overline{\mu}}^{\mathfrak{A}}(a^*)$. Hence, $\mathscr{L}_{\mu}^{\mathfrak{A}}(a)=\mathscr{L}_{\overline{\mu}}^{\mathfrak{A}}(a^*)$ for each $\mu\in W_{\mathfrak{A}}(a)$. Therefore, $DW_{\mathfrak{A}}(a^*)=\left\{(\overline{\mu},r):(\mu,r)\in DW_{\mathfrak{A}}(a)\right\}$.
     \end{proof}
     We now describe the behaviour of the upper boundary of $DW_{\mathfrak{A}}(a)$ under taking adjoints.
     \begin{theorem}
         Let $a\in\mathfrak{A}$. Then
         \[\partial_{+}DW_{\mathfrak{A}}(a^*)=\left\{(\overline{\mu},h_a^{\mathfrak{A}}(\mu)):(\mu,h_a^{\mathfrak{A}}(\mu))\in\partial_{+}DW_{\mathfrak{A}}(a)\right\}.\]
     \end{theorem}
     \begin{proof}
     It is known that $W_{\mathfrak{A}}(a^*)=\overline{W_{\mathfrak{A}}(a)}$ for all $a\in\mathfrak{A}$. Now, we have  
     \begin{eqnarray*}
    \partial_{+}DW_{\mathfrak{A}}(a^*)&=&\left\{(\nu,\ h_a^{{\prime}\mathfrak{A}}(\nu))\in\mathbb{C}\times\mathbb{R}:\nu\in W_{\mathfrak{A}}(a^*),h_a^{{\prime}\mathfrak{A}}(\nu)=\sup\mathscr{L}_{\nu}^{\mathfrak{A}}(a^*)\right\}\\
    &=&\left\{(\overline{\mu},\ h_a^{{\prime}\mathfrak{A}}(\overline{\mu}))\in\mathbb{C}\times\mathbb{R}:\overline{\mu}\in \overline{W_{\mathfrak{A}}(a)},h_a^{{\prime}\mathfrak{A}}(\overline{\mu})=\sup\mathscr{L}_{\overline{\mu}}^{\mathfrak{A}}(a^*)\right\}\\
    &=&\left\{(\overline{\mu},\ h_a^{\mathfrak{A}}({\mu}))\in\mathbb{C}\times\mathbb{R}:\mu\in W_{\mathfrak{A}}(a),h_a^{\mathfrak{A}}(\mu)=\sup\mathscr{L}_{\mu}^{\mathfrak{A}}(a)\right\}\quad\text{(by Theorem~\ref{supL})}.
     \end{eqnarray*} Hence, the result follows.
     \end{proof}
     The following corollary is immediate.
     \begin{corollary}
Let $a\in\mathfrak{A}$. If $W_{\mathfrak{A}}(a)$ is symmetric about the real axis, then
\[
\partial_{+}DW_{\mathfrak{A}}(a)=\partial_{+}DW_{\mathfrak{A}}(a^*).
\]
\end{corollary}

\section{Estimations of the algebraic Davis--Wielandt radius}\label{sec3}

We begin this section by investigating several fundamental properties of the Davis--Wielandt radius for elements within a $C^*$-algebra. The following theorem demonstrates, among other properties, that the algebraic Davis--Wielandt radius fails to define a norm due to its behavior under scalar multiplication.

\begin{theorem}\label{dw-prop}\cite[Theorem~2.1]{hassaouy2024inequalities}
Let $\mathfrak{A}$ be a unital $C^*$-algebra, $a, b \in \mathfrak{A}$, and $f \in \mathfrak{S}(\mathfrak{A})$. Then:
\begin{enumerate}
    \item[\rm (i)] $dw_{\mathfrak{A}}(a) \geq 0$, and $dw_{\mathfrak{A}}(a) = 0$ if and only if $a = 0$.
    \item[\rm (ii)] For any $\alpha \in \mathbb{C}$,
    \[
    dw_{\mathfrak{A}}(\alpha a) 
    \begin{cases}
        \geq |\alpha| dw_{\mathfrak{A}}(a) & \text{if } |\alpha| > 1, \\
        = |\alpha| dw_{\mathfrak{A}}(a) & \text{if } |\alpha| = 1, \\
        \leq |\alpha| dw_{\mathfrak{A}}(a) & \text{if } |\alpha| < 1.
    \end{cases}
    \]
    \item[\rm (iii)] The radius $dw_{\mathfrak{A}}(\cdot)$ satisfies the sharp bounds
    \begin{equation}\label{dwinq}
    \max\{w_{\mathfrak{A}}(a), \|a\|^2\} \leq dw_{\mathfrak{A}}(a) \leq \sqrt{w_{\mathfrak{A}}^2(a) + \|a\|^4}.
    \end{equation}
    \item[\rm (iv)] Assuming $|\alpha| \neq 1$, then $dw_{\mathfrak{A}}(\alpha a) = |\alpha| dw_{\mathfrak{A}}(a)$ if and only if $a = 0$ or $\alpha = 0$.
    \item[\rm (v)] The subadditivity of $dw_{\mathfrak{A}}(\cdot)$ is controlled via
    \[
    dw_{\mathfrak{A}}(a+b) \leq \sqrt{2\big(dw_{\mathfrak{A}}(a) + dw_{\mathfrak{A}}(b)\big) + 4\big(dw_{\mathfrak{A}}(a) + dw_{\mathfrak{A}}(b)\big)^2}.
    \]
\end{enumerate}
\end{theorem}

\begin{remark}
The inequalities established in \eqref{dwinq} are sharp. Consider $a = \boldsymbol{1}$, which yields $dw_{\mathfrak{A}}(a) = \sqrt{w_{\mathfrak{A}}(a)^2 + \|a\|^4} = \sqrt{2}$. Alternatively, letting $\mathfrak{A} = \mathfrak{M}_2(\mathbb{C})$ and choosing the nilpotent matrix $A = \begin{pmatrix} 0 & 1 \\ 0 & 0 \end{pmatrix}$, we find $dw_{\mathfrak{A}}(A) = \max\{w_{\mathfrak{A}}(A), \|A\|^2\} = 1$. Furthermore, the upper bound in \eqref{dwinq} becomes an equality if and only if $a$ is normaloid (see Corollary~\ref{norm corol}).
\end{remark}

Next, we establish that the algebraic Davis--Wielandt radius behaves well under convex combinations.

\begin{theorem}
Let $a, b \in \mathfrak{A}$. Then the mapping $dw_{\mathfrak{A}}(\cdot) : \mathfrak{A} \to \mathbb{R}$ is a convex function.
\end{theorem}

\begin{proof}
Let $a, b \in \mathfrak{A}$ and $\lambda \in [0, 1]$. We use the convexity of the function $\phi(t) = t^{2n}$ for $n \in \mathbb{N}$. By the linearity and properties of states, it is immediate that
\begin{align*}
|f(\lambda a + (1-\lambda) b)| \leq \lambda |f(a)| + (1-\lambda)|f(b)|.
\end{align*}
Setting $c = \lambda a + (1-\lambda)b$ and using the Cauchy-Schwarz inequality for states, we have
\begin{align*}
|f(c^*c)| &= |f(\lambda^2 a^*a + \lambda(1-\lambda)b^*a + \lambda(1-\lambda)a^*b + (1-\lambda)^2 b^*b)| \\
&\leq \lambda^2 f(a^*a) + \lambda(1-\lambda)|f(b^*a)| + \lambda(1-\lambda)|f(a^*b)| + (1-\lambda)^2 f(b^*b) \\
&\leq \lambda^2 f(a^*a) + 2\lambda(1-\lambda)\sqrt{f(a^*a)}\sqrt{f(b^*b)} + (1-\lambda)^2 f(b^*b) \\
&= \left(\lambda\sqrt{f(a^*a)} + (1-\lambda)\sqrt{f(b^*b)}\right)^2 \\
&\leq \lambda f(a^*a) + (1-\lambda) f(b^*b).
\end{align*}
Treating these terms within the Euclidean norm structure of $\mathbb{R}^2$, it follows that
\begin{align*}
(|f(c)|, |f(c^*c)|) &\leq \lambda(|f(a)|, f(a^*a)) + (1-\lambda)(|f(b)|, f(b^*b)),
\end{align*}
and consequently
\begin{align*}
\sqrt{|f(c)|^2 + |f(c^*c)|^2} \leq \lambda\sqrt{|f(a)|^2 + f(a^*a)^2} + (1-\lambda)\sqrt{|f(b)|^2 + f(b^*b)^2}.
\end{align*}
Taking the supremum over all states $f \in \mathfrak{S}(\mathfrak{A})$, yields the desired convexity
\begin{align*}
dw_{\mathfrak{A}}(\lambda a + (1-\lambda)b) \leq \lambda dw_{\mathfrak{A}}(a) + (1-\lambda)dw_{\mathfrak{A}}(b).
\end{align*}
\end{proof}

\begin{lemma}\label{w1}\cite[Theorem~2.1]{hassaouy2024inequalities}
Let $a \in \mathfrak{A}$. Then $w_{\mathfrak{A}}(a, a^*a) = dw_{\mathfrak{A}}(a)$, where for an $n$-tuple $a = (a_1, \dots, a_n) \in \mathfrak{A}^n$, the joint numerical radius is defined by
\[
w_{\mathfrak{A}}(a_1, \dots, a_n) = \sup_{f \in \mathfrak{S}(\mathfrak{A})} \left( \sum_{i=1}^n |f(a_i)|^2 \right)^{\frac{1}{2}}.
\]
\end{lemma}
The next result gives a relation between $dw_{\mathfrak{A}}$ and $w_{\mathfrak{A}}(a)$.
\begin{theorem}\label{conv}
    Let $a\in\mathfrak{A}$. Then $dw_{\mathfrak{A}}(a)=\sqrt{2}w_{\mathfrak{A}}(a)$ if $a$ is self-adjoint and idempotent.
\end{theorem}
  \begin{proof}
Since $a$ is self-adjoint and idempotent, we have $a=a^*=a^2.$ By Lemma~\ref{w1},
\[
dw_{\mathfrak A}(a) = w_{\mathfrak A}(a,a^*a) = w_{\mathfrak A}(a,a^2) = w_{\mathfrak A}(a,a).
\]
Hence,
\begin{align*}
dw_{\mathfrak A}(a)
&= \sup_{f\in \mathfrak S(\mathfrak A)} \Bigl( |f(a)|^2+|f(a)|^2 \Bigr)^{1/2} \\
&= \sqrt{2} \sup_{f\in \mathfrak S(\mathfrak A)} |f(a)| 
= \sqrt{2} w_{\mathfrak A}(a).
\end{align*}
\end{proof}

\begin{example}
The converse of Theorem~\ref{conv} does not hold in general.
Consider $a=-\mathbf{1}$. Then $a$ is self-adjoint, but it is not idempotent.
Further,
\[
w_{\mathfrak A}(a) = \sup_{f\in \mathfrak S(\mathfrak A)} |f(-\mathbf 1)| = 1.
\]
Also,
\begin{align*}
dw_{\mathfrak A}(a)
= \sup_{f\in \mathfrak S(\mathfrak A)} \Bigl( |f(-\mathbf 1)|^2+f(\mathbf 1)^2 \Bigr)^{1/2} = \sqrt{2}.
\end{align*}
Hence,
\[
dw_{\mathfrak A}(a) = \sqrt{2} w_{\mathfrak A}(a),
\]
although $a$ is not idempotent.
\end{example}

\begin{remark}
The above example shows that the converse of Theorem~\ref{conv} fails.
We note that, in the Hilbert space setting, Alomari \cite[Theorem 2.1]{alomari2023davis} claimed that
\[
dw(S)=\sqrt{2}w(S)
\]
if and only if $S$ is a self-adjoint idempotent operator.
The proof of the converse direction in \cite{alomari2023davis} uses the implication
\[
w_e(S,S^*S)=w_e(S,S^2) \quad\Longrightarrow\quad S^*S=S^2,
\]
where $w_e$ denotes the Euclidean operator radius. However, this implication not true in general. Indeed, consider
\[
S=
\begin{pmatrix}
1 & 0 \\
0 & i
\end{pmatrix}.
\]
Then
\[
S^*S=I, \qquad S^2=
\begin{pmatrix}
1 & 0 \\
0 & -1
\end{pmatrix},
\]
and therefore $S^*S\neq S^2$.
Let $x=(\alpha,\beta)\in \mathbb C^2$ with $|x|=1$, and set $t=|\alpha|^2$. Then $|\beta|^2=1-t$.
A straightforward computation gives
\[
|\langle Sx,x\rangle|^2 = t^2+(1-t)^2, \mbox{ and } \langle S^*Sx,x\rangle=1.
\]
Hence,
\[
w_e(S,S^*S)^2 = \sup_{0\le t\le 1} \Bigl( t^2+(1-t)^2+1 \Bigr) = 2.
\]
On the other hand,
\[\langle S^2x,x\rangle=t-(1-t)=2t-1,\] and therefore
\[
w_e(S,S^2)^2 = \sup_{0\le t\le 1} \Bigl( t^2+(1-t)^2+(2t-1)^2 \Bigr) = 2.
\]
Consequently,
\[
w_e(S,S^*S)=w_e(S,S^2).
\]
\end{remark}

\begin{lemma}\label{lemma3}\cite[pp.~75--76]{halmos1982hilbert}
		Let $T\in \mathscr{B}(\mathscr{H})$. Then
		\[
			|\langle Tx, y\rangle|^2\leq \langle|T|x, x\rangle\langle|T^*|y, y\rangle\qquad\ \mathrm{for}\ \mathrm{every}\ x, y\in\mathscr{H}.
		\]
	\end{lemma}
In what follows, we state the following lemma.
\begin{lemma}\label{inequ5}
Let $a\in \mathfrak{A}$, $f\in \mathfrak{S}(\mathfrak{A})$. Then
\begin{equation*}
	|f(a)|^2\leq f(|a|)f(|a^*|).
\end{equation*}
\end{lemma}
\begin{proof} Let $a\in \mathfrak{A}$, $f\in \mathfrak{S}(\mathfrak{A})$. Then
\begin{eqnarray*}
		 |f(a)|^2 
		& =& |\langle \pi(a)\xi, \xi\rangle|^2\quad \text{(by Theorem~\ref{gns})} \\
		& \leq& \langle |\pi(a)|\xi, \xi\rangle\langle |\pi(a^*)|\xi, \xi\rangle\quad \text{(by Lemma~\ref{lemma3})} \\
		& = &\langle \pi(|a|)\xi, \xi\rangle\langle \pi(|a^*|)\xi, \xi\rangle\quad \text{(since from continuous function calculus $|\pi(a)|=\pi(|a|)$)} \\
		& =& f(|a|)f(|a^*|).
\end{eqnarray*}
\end{proof}
    In this section, we establish several inequalities of algebraic Davis--Wielandt radius of $a\in\mathfrak{A}$. To prove the next theorem we need the following lemmas.
\begin{lemma}\label{inequ1}\cite[p.148]{dragomir2007advances}
	For any $x,y,z \in \mathscr{H}$,
	\[
		|\langle x, y\rangle|^2+|\langle x, z\rangle|^2\leq \|x\|^2\sqrt{|\langle y, y\rangle|^2+2|\langle y, z\rangle|^2+|\langle z, z\rangle|^2}.
	\]
\end{lemma}
\begin{lemma}\label{inequ2}\cite[Lemma~2.9]{zamani2020some}
	For any $\alpha,\beta \in \mathbb{C}$, 
	\[
		\sup\limits_{|\gamma|^2+|\delta|^2\leq1}|\gamma\alpha+\delta\beta|^2=|\alpha|^2+|\beta|^2.
	\]
\end{lemma}
\begin{lemma}\label{inequ3}\cite[Lemma~2.10]{zamani2020some}
	Let $S, R \in \mathscr{B}(\mathscr{H})$. For any $\gamma, \delta \in \mathbb{C}$,
	\[
		\|\gamma S+\delta R\|^2\leq (|\gamma|^2+|\delta|^2)\|S^*S+R^*R\|.
	\]
\end{lemma}
\begin{theorem}\label{thm.3.10}
		Let $a\in \mathfrak{A}$, $f\in \mathfrak{S}(\mathfrak{A})$. Then
		\[
			dw_{\mathfrak{A}}^2(a)\leq (w_{\mathfrak{A}}(|a|^4+|a|^8)+2w_{\mathfrak{A}}^2(|a|^2a))^{\tfrac{1}{2}}.
		\]
\end{theorem}
\begin{proof} Let $a\in \mathfrak{A}$, $f\in \mathfrak{S}(\mathfrak{A})$. Then
	\begin{eqnarray*}
			 &&(|f(a)|^2+|f(a^*a)|^2)^2 \\
			& =& (|\langle \pi(a)\xi, \xi\rangle|^2+|\langle \pi(a^*a)\xi, \xi\rangle|^2)^2\quad\text{(by Theorem~\ref{gns})} \\
			& = &(|\langle \xi, \pi(a)\xi\rangle|^2+|\langle \xi, |\pi(a)|^2\xi\rangle|^2)^2 \quad\text{(as, $|\pi(a)|^2=\pi(|a|^2))$} \\
			& \leq &\|\xi\|^4(|\langle \pi(a)\xi, \pi(a)\xi\rangle|^2+2|\langle \pi(a)\xi, |\pi(a)|^2 \xi\rangle|^2+|\langle |\pi(a)|^2\xi, |\pi(a)|^2\xi\rangle|^2) \quad\text{(by Lemma~\ref{inequ1})} \\
			& =& |\langle |\pi(a)|^2\xi, \xi\rangle|^2+|\langle |\pi(a)|^4\xi, \xi\rangle|^2+2|\langle |\pi(a)|^2\pi(a)\xi, \xi\rangle|^2 \\
			& \leq& \sup\limits_{\|\xi\|=1}(|\langle |\pi(a)|^2\xi, \xi\rangle|^2+|\langle |\pi(a)|^4\xi, \xi\rangle|^2)+2|f(|a|^2a)|^2 \\
			&  \leq& \sup\limits_{\|\xi\|=1}(|\langle |\pi(a)|^2\xi, \xi\rangle|^2+|\langle |\pi(a)|^4\xi, \xi\rangle|^2)+2w_{\mathfrak{A}}^2(|a|^2a) \\
			& =&	\sup\limits_{|\gamma|^2+|\delta|^2\leq1}(\sup\limits_{\|\xi\|=1}|\langle(\gamma|\pi(a)|^2+\delta|\pi(a)|^4)\xi, \xi\rangle|^2)+2w_{\mathfrak{A}}^2(|a|^2a) \quad\text{(by Lemma~\ref{inequ2})} \\
			& \leq& \sup\limits_{|\gamma|^2+|\delta|^2\leq1}(\|\gamma|\pi(a)|^2+\delta|\pi(a)|^4\|^2)+2w_{\mathfrak{A}}^2(|a|^2a) \quad\text{(since $\gamma|\pi(a)|^2+\delta|\pi(a)|^4$ is normal}) \\
			& \leq& \sup\limits_{|\gamma|^2+|\delta|^2\leq1}(|\gamma|^2+|\delta|^2)\|(|\pi(a)|^2)^*|\pi(a)|^2+(|\pi(a)|^4)^*|\pi(a)|^4\| 
			 +2w_{\mathfrak{A}}^2(|a|^2a) \quad\text{(by Lemma~\ref{inequ3})} \\
			& \leq& \|\pi((|a|^2)^*(|a|^2))+
			\pi((|a|^4)^*(|a|^4))\|+2w_{\mathfrak{A}}^2(|a|^2a) \\
			&=& \|\pi((|a|^2)^*(|a|^2)+
			(|a|^4)^*(|a|^4))\|+2w_{\mathfrak{A}}^2(|a|^2a) \\
			& \leq& \||a|^4+|a|^8\|+2w_{\mathfrak{A}}^2(|a|^2a) \\
			& = &w_{\mathfrak{A}}(|a|^4+|a|^8)+2w_{\mathfrak{A}}^2(|a|^2a). 
	\end{eqnarray*}
    Thus, \[|f(a)|^2+|f(a^*a)|^2\leq (w_{\mathfrak{A}}(|a|^4+|a|^8)+2w_{\mathfrak{A}}^2(|a|^2a))^{\tfrac{1}{2}}.\]
		If we take the supremum over $f\in \mathfrak{S}(\mathfrak{A})$, we get
		\[
			dw_{\mathfrak{A}}^2(a)\leq (w_{\mathfrak{A}}(|a|^4+|a|^8)+2w_{\mathfrak{A}}^2(|a|^2a))^{\tfrac{1}{2}}.
		\]
\end{proof}
The Crawford number of an element $a$, denoted by $\mathfrak{C}(a)$, is defined as (see \cite{zamani2019characterization})
    \[\mathfrak{C}(a)=\inf\left\{|f(a)|:f\in \mathfrak{S}(\mathfrak{A})\right\}.\] 
The next theorem gives a lower bound for the algebraic Davis--Wielandt radius of $a\in\mathfrak{A}$.
\begin{theorem}\label{thm.3.12}
		Let $a\in \mathfrak{A}$, $f\in \mathfrak{S}(\mathfrak{A})$. Then
		\[dw_{\mathfrak{A}}^2(a)\geq
\dfrac{1}{2}\sup\limits_{\theta\in\mathbb{R}}(w_{\mathfrak{A}}^2(e^{i\theta}a+a^*a)+\mathfrak{C}^2(e^{i\theta}a-a^*a)),
		\] where $\mathfrak{C}(e^{i\theta}a-a^*a)$ is the Crawford number of $e^{i\theta}a-a^*a$.
\end{theorem}
\begin{proof}
	Let $a\in\mathfrak{A}$ and $f\in\mathfrak{S}(\mathfrak{A})$. Then there exists $\theta\in\mathbb{R}$ such that $|f(a)|=e^{i\theta}f(a)$. Hence, we have 
	\begin{eqnarray*}
			 |f(a)|^2+|f(a^*a)|^2 
			& =& \dfrac{1}{2}(|f(a)|+f(a^*a))^2+\dfrac{1}{2}(|f(a)|-f(a^*a))^2 \\
			& = &\dfrac{1}{2}(e^{i\theta}f(a)+f(a^*a))^2+\dfrac{1}{2}(e^{i\theta}f(a)-f(a^*a))^2 \\
			& \geq& \dfrac{1}{2}(f(e^{i\theta}a+a^*a))^2+\dfrac{1}{2}\mathfrak{C}^2(e^{i\theta}a-a^*a) \quad\text{(since, $f(e^{i\theta}a\pm a^*a)\in\mathbb{R})$}.
	\end{eqnarray*}
	Taking the supremum over $f\in \mathfrak{S}(\mathfrak{A})$, we get
	\[
		dw_{\mathfrak{A}}^2(a)\geq \dfrac{1}{2}w_{\mathfrak{A}}^2(e^{i\theta}a+a^*a)+\dfrac{1}{2}\mathfrak{C}^2(e^{i\theta}a-a^*a).
	\]
	This holds for all $\theta\in\mathbb{R}$. Hence,
		\[
		dw_{\mathfrak{A}}^2(a)\geq \dfrac{1}{2}\sup\limits_{\theta\in\mathbb{R}}(w_{\mathfrak{A}}^2(e^{i\theta}a+a^*a)+\mathfrak{C}^2(e^{i\theta}a-a^*a)).
	\]
\end{proof}
Next, we derive an upper bound for the Davis--Wielandt radius of an element in $C^*$-algebra by means of the following lemma.
\begin{lemma}\label{inequ4}\cite[Lemma~2.1]{dragomir2015reverses}
	For any $x,y\in\mathscr{H}$ and any $\lambda\in\mathbb{C}$,
	\[
		\|x\|^2\|y\|^2-|\langle x, y\rangle|^2=	\|x-\lambda y\|^2\|y\|^2-|\langle x-\lambda y, y\rangle|^2.
	\]
\end{lemma}
\begin{theorem}\label{thm.3.14}
	Let $a\in \mathfrak{A}$, $f\in \mathfrak{S}(\mathfrak{A})$ and $\lambda\in\mathbb{C}$. Then
	\[
		dw_{\mathfrak{A}}^2(a)\leq \inf\limits_{\lambda\in\mathbb{C}}\left(2\|\Re(\lambda)\Re(a)+\Im(\lambda)\Im(a)\|+\|a^*a-2\Re(\bar{\lambda}a)\|^2+2\|\Re(\overline{\lambda}a)\|-|\lambda|^2+w_{\mathfrak{A}}^2(a-\lambda\boldsymbol{1})\right).
	\] In particular, $dw_{\mathfrak{A}}(a)\leq \sqrt{w_{\mathfrak{A}}^2(a)+\|a\|^4}$.
    \end{theorem}
	\begin{proof}
		Let $a\in\mathfrak{A}$ and $f\in\mathfrak{S}(\mathfrak{A})$. Then from Theorem~\ref{gns} there exists a representation $\pi$ of $\mathfrak{A}$ and a unit vector $\xi\in\mathscr{H}$ such that $f(a)=\langle\pi(a)\xi,\xi\rangle$. We have $a=\Re(a)+i\Im(a)$
		$\implies \pi(a)=\Re(\pi(a))+i\Im(\pi(a))$.
		Also, let $\lambda\in\mathbb{C}$. Then
		from Lemma~\ref{inequ4}, we have
		\begin{eqnarray}\label{new eqn1}
			\|\pi(a)\xi\|^2\|\xi\|^2-|\langle\pi(a)\xi, \xi\rangle|^2=	\|\pi(a)\xi-\lambda\xi\|^2\|\xi\|^2-|\langle\pi(a)\xi-\lambda\xi, \xi\rangle|^2.
		\end{eqnarray}
		Hence,
		\begin{eqnarray*}
				& &\|\pi(a)\xi\|^2 \\
				& =& (\langle\Re(\pi(a))\xi, \xi\rangle)^2-(\langle\Re(\pi(a)-\lambda I)\xi, \xi\rangle)^2+(\langle\Im(\pi(a))\xi, \xi\rangle)^2-(\langle\Im(\pi(a)-\lambda I)\xi, \xi\rangle)^2+\|\pi(a)\xi-\lambda\xi\|^2 \\
				& = &\langle(2\Re(\pi(a))-\Re(\lambda)I)\xi, \xi\rangle\langle\Re(\lambda)\xi, \xi\rangle+\langle(2\Im(\pi(a))-\Im(\lambda)I)\xi, \xi\rangle\langle\Im(\lambda)\xi, \xi\rangle+\|\pi(a)\xi-\lambda\xi\|^2 \\
				& =& 2\Re(\lambda)\langle\Re(\pi(a))\xi, \xi\rangle+2\Im(\lambda)\langle\Im(\pi(a))\xi, \xi\rangle-(\Re(\lambda))^2-(\Im(\lambda))^2+\|\pi(a)\xi-\lambda\xi\|^2 \\
				& = & 2\Re(\lambda)\langle\Re(\pi(a))\xi, \xi\rangle+2\Im(\lambda)\langle\Im(\pi(a))\xi, \xi\rangle-|\lambda|^2+\|\pi(a)\xi-\lambda\xi\|^2 \\ 
				& =&  2(\Re(\lambda)\langle\Re(\pi(a))\xi, \xi\rangle+\Im(\lambda)\langle\Im(\pi(a))\xi, \xi\rangle)+\langle(\pi(a^*a)-2\Re(\overline{\lambda}\pi(a)))\xi, \xi\rangle \\
				& \leq& 2\|\Re(\lambda)\Re(a)+\Im(\lambda)\Im(a)\|+\|a^*a-2\Re(\overline{\lambda}a)\|. 
		\end{eqnarray*}
		Again, from equation~\eqref{new eqn1}, we get
		\begin{eqnarray*}
				|\langle\pi(a)\xi, \xi\rangle|^2 
				& = &	\|\pi(a)\xi\|^2\|\xi\|^2-\|\pi(a)\xi-\lambda\xi\|^2\|\xi\|^2+|\langle\pi(a)\xi-\lambda\xi, \xi\rangle|^2 \\
				& = &2\langle\Re(\overline{\lambda}\pi(a))\xi, \xi\rangle-|\lambda|^2+|\langle\pi(a)\xi-\lambda\xi, \xi\rangle|^2 \\
				& = &2\langle\pi(\Re(\overline{\lambda}a))\xi, \xi\rangle-|\lambda|^2+|\langle\pi(a-\lambda\boldsymbol{1})\xi, \xi\rangle|^2 \\
				& \leq &2\|\pi\|\|\Re(\overline{\lambda}a)\|-|\lambda|^2+|f(a-\lambda\boldsymbol{1})|^2 \\
				& \leq& 2\|\Re(\overline{\lambda}a)\|-|\lambda|^2+w_{\mathfrak{A}}^2(a-\lambda\boldsymbol{1}) \quad\text{(since $\|\pi\|\leq 1$)}.
		\end{eqnarray*}
		Thus, 
		\[
			|\langle\pi(a)\xi, \xi\rangle|^2+|\langle\pi(a^*a)\xi, \xi\rangle|^2\leq 2\|\Re(\overline{\lambda}a)\|-|\lambda|^2+w_{\mathfrak{A}}^2(a-\lambda\boldsymbol{1})+2\|\Re(\lambda)\Re(a)+\Im(\lambda)\Im(a)\|+\|a^*a-2\Re(\overline{\lambda}a)\|\]
            Hence, we have \[
            |f(a)|^2+f(a^*a)^2\leq 2\|\Re(\overline{\lambda}a)\|-|\lambda|^2+w_{\mathfrak{A}}^2(a-\lambda\boldsymbol{1})+2\|\Re(\lambda)\Re(a)+\Im(\lambda)\Im(a)\|+\|a^*a-2\Re(\overline{\lambda}a)\|.
		\]
			Taking the supremum over $f\in \mathfrak{S}(\mathfrak{A})$ and the infimum over all $\lambda\in\mathbb{C}$, the required inequality holds.

            \noindent If we take $\lambda=0$, then we get $dw_{\mathfrak{A}}(a)\leq \sqrt{w_{\mathfrak{A}}^2(a)+\|a\|^4}$.
	\end{proof}
In the next result, we deduce an upper bound for the algebraic Davis--Wielandt radius of the sum of two elements.
\begin{theorem}\label{thm.3.16}
Let $a,b\in \mathfrak{A}$, $f\in \mathfrak{S}(\mathfrak{A})$. Then
\[
	dw_{\mathfrak{A}}(a+b)\leq dw_{\mathfrak{A}}(a)+dw_{\mathfrak{A}}(b)+\|a^*b+b^*a\|.
\] 
\end{theorem}
\begin{proof} Let $a, b\in \mathfrak{A}$ and $f\in \mathfrak{S}(\mathfrak{A})$. Then
	\begin{eqnarray*}
			DW_{\mathfrak{A}}(a+b)
			& = &\{(f(a+b), f((a+b)^*(a+b))):f\in\mathfrak{S}(\mathfrak{A})\}\\
			& = &\{(f(a), f(a^*a))+(f(b), f(b^*b))+(0, f(a^*b+b^*a)):f\in\mathfrak{S}(\mathfrak{A})\}.
	\end{eqnarray*}
	Hence, $DW_{\mathfrak{A}}(a+b)\subseteq DW_{\mathfrak{A}}(a)+DW_{\mathfrak{A}}(b)+\mathfrak{N}$, where $\mathfrak{N}=\{(0, f(a^*b+b^*a)):f\in\mathfrak{S}(\mathfrak{A})\}$. Then one can easily get the required inequality.
\end{proof}
\begin{corollary}
    Let $a,b\in\mathfrak{A}$ be such that $a^*b+b^*a=0$. Then
    \[dw_{\mathfrak{A}}(a+b)\leq dw_{\mathfrak{A}}(a)+dw_{\mathfrak{A}}(b).\]
\end{corollary}
The next theorem provides new lower and upper estimates for the Davis--Wielandt radius of an element in a $C^*$-algebra.
\begin{theorem}\label{thm.3.17}
	Let $a\in \mathfrak{A}$, $f\in \mathfrak{S}(\mathfrak{A})$. Then
	\begin{eqnarray*}
				&& \max\{w_{\mathfrak{A}}(\Re(a)+i|a|^2), w_{\mathfrak{A}}(\Im(a)+i|a|^2)\}\leq dw_{\mathfrak{A}}(a) \\
				& \leq&\min\left\{\sqrt{w_{\mathfrak{A}}^2(\Re(a)+i|a|^2)+\|\Im(a)\|^2}, \sqrt{w_{\mathfrak{A}}^2(\Im(a)+i|a|^2)+\|\Re(a)\|^2}\right\}.
	\end{eqnarray*}
\end{theorem}
\begin{proof}
	We have $a=\Re(a)+i\Im(a)$.
	Now,
	\begin{eqnarray*}
				|f(a)|^2+f(a^*a)^2 
				& = & |f(\Re(a))|^2+ |f(\Im(a))|^2+f(|a|^2)^2\\
				& =& |f(\Re(a)+i|a|^2)|^2 +|f(\Im(a))|^2 \\
				& \leq &w_{\mathfrak{A}}^2 (\Re(a)+i|a|^2) +\|\Im(a)\|^2.
	\end{eqnarray*}
	Taking the supremum over $f\in \mathfrak{S}(\mathfrak{A})$, we obtain
	\begin{eqnarray}\label{min1}
		dw_{\mathfrak{A}}^2(a)\leq w_{\mathfrak{A}}^2 (\Re(a)+i|a|^2) +\|\Im(a)\|^2.
	\end{eqnarray} 
	Similarly, proceeding as above, we get \begin{eqnarray}\label{min2}
		dw_{\mathfrak{A}}^2(a)\leq w_{\mathfrak{A}}^2 (\Im(a)+i|a|^2) +\|\Re(a)\|^2.
	\end{eqnarray} Combining inequalities~\eqref{min1} and \eqref{min2} yields
    \[dw_{\mathfrak{A}}(a)\leq\min\left\{\sqrt{w_{\mathfrak{A}}^2(\Re(a)+i|a|^2)+\|\Im(a)\|^2}, \sqrt{w_{\mathfrak{A}}^2(\Im(a)+i|a|^2)+\|\Re(a)\|^2}\right\}.\]
	Now,\begin{eqnarray*}
		|f(a)|^2+|f(a^*a)|^2&=& |f(\Re(a)+i|a|^2)|^2 +|f(\Im(a))|^2\\
		&\geq& |f(\Re(a)+i|a|^2)|^2.
        \end{eqnarray*}
		Taking the supremum over $f\in \mathfrak{S}(\mathfrak{A})$, we get
			\begin{eqnarray}\label{max5}
			dw_{\mathfrak{A}}(a)\geq w_{\mathfrak{A}}(\Re(a)+i|a|^2).
		\end{eqnarray}
			In a similar way, we can get
			\begin{eqnarray}\label{max6}
            dw_{\mathfrak{A}}(a)\geq w_{\mathfrak{A}}(\Im(a)+i|a|^2).
		\end{eqnarray} Inequalities~\eqref{max5} and \eqref{max6} together imply \[\max\left\{w_{\mathfrak{A}}(\Re(a)+i|a|^2), w_{\mathfrak{A}}(\Im(a)+i|a|^2)\right\}\leq dw_{\mathfrak{A}}(a).\]
\end{proof}
\begin{corollary}
    Let $a\in\mathfrak{A}$ be such that $\Im(a)=(\Re(a))^2$. Then \[w_{\mathfrak{A}}(a)=\|\Re(a)\|\sqrt{1+\|\Re(a)\|^2}.\]
\end{corollary}
\begin{proof}
   Let $b\in\mathfrak{A}$ be self-adjoint. Then from Theorem~\ref{thm.3.17} it follows that 
   \[\max\left\{w_{\mathfrak{A}}(b+ib^2),w_{\mathfrak{A}}(ib^2)\right\}\leq dw_{\mathfrak{A}}(b)\leq\min\left\{w_{\mathfrak{A}}(b+ib^2),\sqrt{w_{\mathfrak{A}}^2(ib^2),\|b\|^2}\right\}.\] From this, we can say $dw_{\mathfrak{A}}(b)=w_{\mathfrak{A}}(b+ib^2)$. Also, we have $dw_{\mathfrak{A}}(b)=\sqrt{\|b\|^2+\|b\|^4}$. Thus, $w_{\mathfrak{A}}(b+ib^2)=\sqrt{\|b\|^2+\|b\|^4}$. Now, taking $b=\Re(a)$ completes the proof.
\end{proof}

\begin{remark}
    Theorem~\ref{thm.3.10}, Theorem~\ref{thm.3.12}, Theorem~\ref{thm.3.14}, Theorem~\ref{thm.3.16}, and Theorem \ref{thm.3.17} extend the corresponding upper and lower bounds for the Davis--Wielandt radius established in several research articles (see \cite{zamani2020some,bhunia2021bounds,bhunia2023further}) for bounded linear operators to a more generalized setting of unital $C^*$-algebra.
\end{remark}

We now establish a new upper bound for the algebraic Davis--Wielandt radius of the sum of $k$ elements.

\begin{lemma}\label{lemma}\cite{vasic1971some}
	Suppose $\alpha_1,\alpha_2,\cdots,\alpha_k$ are $k$ positive numbers. Then
	\[
	\left (\sum_{i=1}^{k}\alpha_i\right )^n\leq k^{n-1}\sum_{i=1}^{k}\alpha_i^n\quad \mathrm{for}\ \mathrm{every}\ n\geq 1.
	\]
\end{lemma}
The following inner product inequality (the Buzano's inequality \cite{buzano1974generalizzazione}) is required to prove the next result.
\begin{lemma}\label{lemma1}
	Let $x,y,z\in \mathscr{H}$ with $\|z\|=1$. Then
	\[
		2|\langle x, z\rangle\langle z, y\rangle|\leq |\langle x, y\rangle|+\|x\|\|y\|.
	\]
\end{lemma}
We now estimate an upper bound for the algebraic Davis--Wielandt radius of the sum of $k$ elements of $\mathfrak{A}$.
\begin{theorem}\label{thm.3.22}
	Let $a_i\in\mathfrak{A}$ for all $i\in\{1, 2,\cdots, k\},\ f\in\mathfrak{S}(\mathfrak{A})$. Then for all $\ n\geq1$,
	\[
		dw_{\mathfrak{A}}^{4n}\left(\sum_{i=1}^{k}a_i\right)\leq 2^{2n-2}k^{4n-1}\sum_{i=1}^{k}w_{\mathfrak{A}}(|a_i|^{2n}|a_i^*|^{2n})+2^{2n-3}k^{4n-1}\left\|\sum_{i=1}^{k}(|a_i|^{4n}+|a_i^*|^{4n})\right\|+2^{2n-1}k^{8n-1}\left\|\sum_{i=1}^{k}|a_i|^{8n}\right\|.
	\]
\end{theorem}
\begin{proof}
	Let $a_i\in\mathfrak{A}$ for all $i\in\{1, 2,\cdots, k\},\ f\in\mathfrak{S}(\mathfrak{A})$. By the convexity of $\phi(t)=t^{2n}$, $n\geq1$, we get
	\[
		\left(\left|f\left(\sum_{i=1}^{k}a_i\right)\right|^2+\left|f\left(\left|\sum_{i=1}^{k}a_i\right|^2\right)\right|^2\right)^{2n}\leq \dfrac{2^{2n}}{2}\left[\left|f\left (\sum_{i=1}^{k}a_i\right )\right|^{4n}+\left|f\left (\left|\sum_{i=1}^{k}a_i\right|^2\right )\right|^{4n}\right].
	\] Now,
	\begin{eqnarray*}
		& &\left|f\left (\sum_{i=1}^{k}a_i\right )\right|^{4n}+\left|f\left (\left|\sum_{i=1}^{k}a_i\right|^2\right )\right|^{4n} \\
		& \leq&\left (\sum_{i=1}^{k}\left|f(a_i)\right|\right )^{4n}+f\left (\left|\sum_{i=1}^{k}a_i\right|^{8n}\right ) \\
		& \leq& k^{4n-1}\sum_{i=1}^{k}|f(a_i)|^{4n}+f\left (\left (\sum_{i=1}^{k}|a_i|\right )^{8n}\right ) \quad\text{(by Lemma~\ref{lemma})} \\
		&\leq& k^{4n-1}\sum_{i=1}^{k}\left(f(|a_i|^{2n})f(|a_i^*|^{2n})\right)+f\left (k^{8n-1}\sum_{i=1}^{k}|a_i|^{8n}\right ) \quad\text{(by Lemmas~\ref{inequ5}, \ref{nw lm1} and \ref{lemma})} \\
		& \leq& k^{4n-1}\sum_{i=1}^{k}\langle\pi(|a_i|^{2n})\xi, \xi\rangle\langle\xi, \pi(|a_i^*|^{2n})\xi\rangle+k^{8n-1}f\left (\sum_{i=1}^{k}|a_i|^{8n}\right ) \quad\text{(by Theorem~\ref{gns})}\\
		& \leq& \dfrac{k^{4n-1}}{2}\sum_{i=1}^{k}\left(|\langle\pi(|a_i|^{2n})\xi, \pi(|a_i^*|^{2n})\xi\rangle|+\|\pi(|a_i|^{2n})\xi\|\|\pi(|a_i^*|^{2n})\xi\|\right )+k^{8n-1}f\left (\sum_{i=1}^{k}|a_i|^{8n}\right )\quad\text{(by Lemma~\ref{lemma1})} \\
		&  = &\dfrac{k^{4n-1}}{2}\sum_{i=1}^{k}\left(|\langle\pi(|a_i|^{2n}|a_i^*|^{2n})\xi, \xi\rangle|+\|\pi(|a_i|^{2n})\xi\|\|\pi(|a_i^*|^{2n})\xi\|\right)+k^{8n-1}f\left (\sum_{i=1}^{k}|a_i|^{8n}\right ) \\
		& \leq & \dfrac{k^{4n-1}}{2}\sum_{i=1}^{k}f(|a_i|^{2n}|a_i^*|^{2n})+\dfrac{k^{4n-1}}{4}\sum_{i=1}^{k}|f(|a_i|^{4n}+|a_i^*|^{4n})|+k^{8n-1}f\left (\sum_{i=1}^{k}|a_i|^{8n}\right )  \quad\text{(as, AM $\geq$ GM)} \\
		& \leq& \dfrac{k^{4n-1}}{2}\sum_{i=1}^{k}w_{\mathfrak{A}}(|a_i|^{2n}|a_i^*|^{2n})+\dfrac{k^{4n-1}}{4}\left\|\sum_{i=1}^{k}(|a_i|^{4n}+|a_i^*|^{4n})\right\|+k^{8n-1}\left\|\sum_{i=1}^{k}|a_i|^{8n}\right\|.
	\end{eqnarray*}
	Taking the supremum over $f\in \mathfrak{S}(\mathfrak{A})$, we have
	\begin{eqnarray*}
		&& dw_{\mathfrak{A}}^{4n}\left(\sum_{i=1}^{k}a_i\right) \\
		& \leq& 2^{2n-1}\left[\dfrac{k^{4n-1}}{2}\sum_{i=1}^{k}w_{\mathfrak{A}}(|a_i|^{2n}|a_i^*|^{2n})+\dfrac{k^{4n-1}}{4}\left\|\sum_{i=1}^{k}(|a_i|^{4n}+|a_i^*|^{4n})\right\|+k^{8n-1}\left\|\sum_{i=1}^{k}|a_i|^{8n}\right\|\right] \\
		& =&  2^{2n-2}k^{4n-1}\sum_{i=1}^{k}w_{\mathfrak{A}}(|a_i|^{2n}|a_i^*|^{2n})+2^{2n-3}k^{4n-1}\left\|\sum_{i=1}^{k}(|a_i|^{4n}+|a_i^*|^{4n})\right\|+2^{2n-1}k^{8n-1}\left\|\sum_{i=1}^{k}|a_i|^{8n}\right\|.
	\end{eqnarray*}
\end{proof}
\subsection*{Comparability of bounds}
At the end of this section, we would like to study the comparability of upper bounds for algebraic Davis--Wielandt radii of elements. We illustrate numerical examples to show that, in general, those upper bounds obtained by us are incomparable. We list a few examples by taking the $C^*$-algebra $\mathfrak{A}=\mathfrak{M}_2(\mathbb{C})$, the algebra of $2\times2$ complex matrices, and considering $A_1=\begin{pmatrix}
0 & 1\\
0 & 0
\end{pmatrix}$, $A_2=\begin{pmatrix}
0 & 2\\
0 & 0
\end{pmatrix}$, $A_3=\begin{pmatrix}
0 & 1\\
2 & 0
\end{pmatrix}$, $A_4=\begin{pmatrix}
1 & 1\\
0 & 1
\end{pmatrix}$, $A_5=\begin{pmatrix}
1 & 0\\
0 & -1
\end{pmatrix}$.

\begin{table}[H]
\centering
\small
\renewcommand{\arraystretch}{1}
\caption{Comparison of upper bounds}\label{tab:upperbounds}
\begin{tabular}{@{}lccccc@{}}
\hline
\multicolumn{6}{c}{Upper bounds for} \\
 & $dw_{\mathfrak{A}}^2(A_1)$ & $dw_{\mathfrak{A}}^2(A_2)$ & $dw_{\mathfrak{A}}^2(A_3)$ & $dw_{\mathfrak{A}}^2(A_4)$ & $dw_{\mathfrak{A}}^2(A_5)$\\
\hline
Theorem~\ref{thm.3.10}    & 1.414 & $\mathbf{16.4924}$ & 17.0294 & $\mathbf{9.0970}$ & $\mathbf{2}$\\
Theorem~\ref{thm.3.14}    & $\mathbf{1.25}$     & 17     & 18.25    & 9.104 & $\mathbf{2}$\\
Theorem~\ref{thm.3.17}    & $\mathbf{1.25}$  & 17     & $\mathbf{16.25}$ & 9.104 & $\mathbf{2}$\\
Theorem~\ref{thm.3.22}  & 1.58   & 22.8035 & 22.902 & 10.3178 & $\mathbf{2}$\\
\hline
\end{tabular}
\end{table}

    \section{Norm-Parallelism and the Algebraic Davis--Wielandt Radius}\label{sec4}

In this section, we investigate the relationship between norm-parallelism and the Davis--Wielandt radius within the structural framework of unital $C^*$-algebras. The algebraic Davis--Wielandt shell $DW_{\mathfrak{A}}(a)$ encodes refined geometric and structural information about an element $a \in \mathfrak{A}$ that extends well beyond the classical numerical range. Utilizing the baseline inequalities of Section \ref{sec3}, we show that norm-parallelism to the identity element $\mathbf{1}$ can be completely characterized by an elegant, explicit formula for $dw_{\mathfrak{A}}(a)$. Along the way, we demonstrate that this radius criteria serves as a bridge to characterizing normaloid elements, effectively linking spatial norm-attainment with state-space properties on the powers of the elements. We begin by recalling the formal definition of norm-parallelism and reviewing some of its fundamental properties.

\begin{definition}
Let $\mathfrak{A}$ be a unital $C^*$-algebra. Two elements $a, b \in \mathfrak{A}$ are said to be \textit{norm-parallel}, denoted by $a \parallel b$, if there exists a scalar $\lambda \in \mathbb{T}$ such that
\[
\|a + \lambda b\| = \|a\| + \|b\|.
\]
\end{definition}

In this context, the Daugavet equation $\|a + \boldsymbol{1}\| = \|a\| + 1$ arises as a special case of norm-parallelism. It is worth noting that norm-parallelism is a symmetric and homogeneous relation; however, it does not satisfy transitivity relation in general (see \cite{zamani2016norm}). Although linear dependence clearly implies norm-parallelism, the converse does not hold. Characterizations of norm-parallelism in the broader setting of Hilbert $C^*$-modules can be found in \cite{zamani2015exact,zamani2016norm}.

Throughout this section, we assume that $a$ is a non-zero element of $\mathfrak{A}$. To establish our main results, we shall frequently utilize \cite[Corollary~4.4]{zamani2015exact}, which we restate below for completeness.

\begin{lemma}\label{para cor} \cite[Corollary~4.4]{zamani2015exact}
Let $\mathfrak{A}$ be a $C^*$-algebra, and let $a, b \in \mathfrak{A}$. Then the following statements are equivalent:
\begin{enumerate}
    \item[(i)] $a \parallel b$.
    \item[(ii)] There exists a state $f \in \mathfrak{S}(\mathfrak{A})$ and a scalar $\lambda \in \mathbb{T}$ such that $f(a^*b) = \lambda \|a\| \|b\|$.
    \item[(iii)] There exists a Hilbert space $\mathscr{H}$, a representation $\pi : \mathfrak{A} \to \mathscr{B}(\mathscr{H})$, a unit vector $\xi \in \mathscr{H}$, and a scalar $\lambda \in \mathbb{T}$ such that $\|\pi(a)\xi\| = \|a\|$ and $\langle \pi(a)\xi, \pi(b)\xi \rangle = \lambda \|a\| \|b\|$.
\end{enumerate}
\end{lemma}


The spatial characterization of elements via the Gelfand--Naimark--Segal (GNS) construction allows us to transport problems regarding $C^*$-algebras into the setting of bounded linear operators on a Hilbert space. To this end, we establish an operator-theoretic result concerning operators that simultaneously attain their norms and numerical radii via the same unit vector. This foundational observation serves as a connection between norm-parallelism and the Davis--Wielandt radius.

\begin{theorem}\label{op thm}
Let $T \in \mathscr{B}(\mathscr{H})$. Suppose there exists a unit vector $x \in \mathscr{H}$ such that 
\[
|\langle Tx, x \rangle| = w(T) \quad \text{and} \quad \|Tx\| = \|T\|.
\]
Then $w(T) = \|T\|$.
\end{theorem}

\begin{proof}
Let $\alpha = \langle Tx, x \rangle$. Since $|\alpha| = w(T)$, there exists $\lambda \in \mathbb{T}$ such that $\lambda \alpha = w(T)$. By replacing $T$ with $\lambda T$, we preserve both the norm and the numerical radius, i.e., $\|\lambda T\| = \|T\|$ and $w(\lambda T) = w(T)$. Thus, without loss of generality, we may assume that $\alpha = w(T) > 0$. 

We decompose $Tx$ as 
\[
Tx = \alpha x + u, \quad \text{where } u \perp x.
\]
Consequently, the norm of $T$ satisfies
\[
\|T\|^2 = \|Tx\|^2 = \|\alpha x + u\|^2 = \|\alpha x\|^2 + \|u\|^2 = \alpha^2 + \|u\|^2.
\]
We aim to show that $u = 0$. Seeking a contradiction, suppose that $u \neq 0$. Let $c = \|u\| > 0$ and define the unit vector $v = \frac{u}{c}$, so that $u = cv$ with $v \perp x$. Let $z = \langle Tv, x \rangle$. For any $\theta \in \mathbb{R}$ and sufficiently small $t \in \mathbb{R}$, define the unit vector
\[
y_{t,\theta} = \frac{x + t e^{i\theta} v}{\sqrt{1 + t^2}}.
\]
Clearly, $\|y_{t,\theta}\| = 1$. Consider the function $p_{\theta}(t) = \langle Ty_{t,\theta}, y_{t,\theta} \rangle$. Since $y_{t,\theta}$ is a unit vector, we have $|p_{\theta}(t)| \leq w(T) = \alpha$. Moreover, $p_{\theta}(0) = \langle Tx, x \rangle = \alpha$. Noting that $\langle Tx, v \rangle = \langle \alpha x + cv, v \rangle = c$, a direct computation yields
\begin{align*}
    p_{\theta}(t) &= \frac{\langle T(x + t e^{i\theta} v), x + t e^{i\theta} v \rangle}{1 + t^2} \\
    &= \frac{\langle Tx, x \rangle + t e^{i\theta} \langle Tv, x \rangle + t e^{-i\theta} \langle Tx, v \rangle + t^2 \langle Tv, v \rangle}{1 + t^2} \\
    &= \alpha + t(e^{i\theta} z + e^{-i\theta} c) + O(t^2) \quad \text{as } t \to 0,
\end{align*}
where we used the expansion $\frac{1}{1+t^2} = 1 + O(t^2)$. It follows that
\[
|p_{\theta}(t)|^2 = \alpha^2 + 2\alpha t \Re\left(e^{i\theta} z + e^{-i\theta} c\right) + O(t^2).
\]
Employing the Taylor expansion $\sqrt{1+s} = 1 + \frac{s}{2} + O(s^2)$, we obtain
\[
|p_{\theta}(t)| = \alpha \sqrt{1 + \frac{2t}{\alpha} \Re\left(e^{i\theta} z + e^{-i\theta} c\right) + O(t^2)} = \alpha + t \Re\left(e^{i\theta} z + e^{-i\theta} c\right) + O(t^2).
\]
Since $|p_{\theta}(t)| \leq \alpha$ for all sufficiently small $t$, the coefficient of $t$ must vanish identically. Otherwise, a suitable choice of the sign of $t$ would yield $|p_{\theta}(t)| > \alpha$, a contradiction. Therefore, $\Re\left(e^{i\theta} z + e^{-i\theta} c\right) = 0$ for all $\theta \in \mathbb{R}$. This implies $z = -c$, and hence $\langle Tv, x \rangle = -c$.

Because $\|Tx\| = \|T\|$ with $\|x\| = 1$, the vector $x$ maximizes the norm of $T$, which implies $T^* T x = \|T\|^2 x$. Utilizing this relation, we find
\[
0 = \langle \|T\|^2 x, v \rangle = \langle T^* T x, v \rangle = \langle Tx, Tv \rangle = \langle \alpha x + cv, Tv \rangle = \alpha \langle x, Tv \rangle + c \langle v, Tv \rangle = -\alpha c + c \langle Tv, v \rangle.
\]
Since $c > 0$, we deduce that $\langle Tv, v \rangle = \alpha$. 

Now, consider the compression of $T$ to the two-dimensional subspace spanned by $\{x, v\}$. The corresponding compression matrix $A$ is given by
\[
A = \begin{pmatrix}
    \langle Tx, x \rangle & \langle Tv, x \rangle \\
    \langle Tx, v \rangle & \langle Tv, v \rangle
\end{pmatrix}
= \begin{pmatrix}
    \alpha & -c \\
    c & \alpha
\end{pmatrix}.
\]
By the properties of compressions, $w(A) \leq w(T) = \alpha$. However, taking the unit vector $\beta = \frac{1}{\sqrt{2}} \begin{pmatrix} 1 \\ i \end{pmatrix}$, we find
\[
\langle A\beta, \beta \rangle = \frac{1}{2} \begin{pmatrix} 1 & -i \end{pmatrix} \begin{pmatrix} \alpha & -c \\ c & \alpha \end{pmatrix} \begin{pmatrix} 1 \\ i \end{pmatrix} = \frac{1}{2} [(\alpha - ic) - i(c + i\alpha)] = \alpha - ic.
\]
Thus, $|\langle A\beta, \beta \rangle| = \sqrt{\alpha^2 + c^2}$, implying that $w(A) \geq \sqrt{\alpha^2 + c^2} > \alpha$, which contradicts $w(A) \leq \alpha$. 

Consequently, $u = 0$, meaning $Tx = \alpha x$. Finally,
\[
\|T\| = \|Tx\| = \|\alpha x\| = \alpha = |\langle Tx, x \rangle| = w(T).
\]
\end{proof}

We are now in a position to present our first main connection between norm-parallelism to the identity and the algebraic Davis--Wielandt radius.

\begin{theorem}\label{norm thm}
Let $a \in \mathfrak{A}$. Then the following conditions are equivalent:
\begin{enumerate}
    \item[(i)] $a \parallel \boldsymbol{1}$.
    \item[(ii)] $dw_{\mathfrak{A}}(a) = \sqrt{w_{\mathfrak{A}}^2(a) + \|a\|^4}$.
\end{enumerate}
\end{theorem}

\begin{proof}
(i) $\implies$ (ii): Let $a \parallel \boldsymbol{1}$. By Lemma~\ref{para cor}, there exists a Hilbert space $\mathscr{H}$, a representation $\pi : \mathfrak{A} \to \mathscr{B}(\mathscr{H})$, a unit vector $\xi \in \mathscr{H}$, and a scalar $\lambda \in \mathbb{T}$ such that
\[
\|\pi(a)\xi\|^2 = \langle \pi(a^*a)\xi, \xi \rangle = \|a\|^2 \quad \text{and} \quad \langle \pi(a)\xi, \xi \rangle = \lambda \|a\|.
\]
By the Gelfand--Naimark--Segal (GNS) construction, there exists a state $f \in \mathfrak{S}(\mathfrak{A})$ such that
\[
f(a^*a) = \langle \pi(a^*a)\xi, \xi \rangle = \|a\|^2 \quad \text{and} \quad f(a) = \langle \pi(a)\xi, \xi \rangle = \lambda \|a\|.
\]
It follows that $w_{\mathfrak{A}}(a) \geq |f(a)| = \|a\|$, and by combining this with inequality \eqref{*norm}, we obtain $w_{\mathfrak{A}}(a) = \|a\|$. From the definition of the algebraic Davis--Wielandt radius, we have
\begin{align}\label{equal}
    dw_{\mathfrak{A}}(a) &\geq \sqrt{|f(a)|^2 + f(a^*a)^2} \nonumber \\
    &= \sqrt{\|a\|^2 + \|a\|^4} \nonumber \\
    &= \sqrt{w_{\mathfrak{A}}^2(a) + \|a\|^4}.
\end{align}
Matching this with the upper bound provided by inequality \eqref{dwinq}, we conclude that $dw_{\mathfrak{A}}(a) = \sqrt{w_{\mathfrak{A}}^2(a) + \|a\|^4}$.

(ii) $\implies$ (i): Assume that $dw_{\mathfrak{A}}(a) = \sqrt{w_{\mathfrak{A}}^2(a) + \|a\|^4}$. By the definition of $dw_{\mathfrak{A}}(a)$, there exists a net of states $\{f_{\alpha}\}$ in $\mathfrak{S}(\mathfrak{A})$ such that
\[
\sqrt{|f_{\alpha}(a)|^2 + f_{\alpha}(a^*a)^2} \longrightarrow \sqrt{ w_{\mathfrak{A}}^2(a) + \|a\|^4},
\]
which implies
\begin{equation}\label{new1}
|f_{\alpha}(a)| \longrightarrow w_{\mathfrak{A}}(a) \quad \text{and} \quad f_{\alpha}(a^*a) \longrightarrow \|a\|^2.
\end{equation}
Since the state space $\mathfrak{S}(\mathfrak{A})$ is weak$^*$-compact, we can extract a weak$^*$-convergent subnet $\{f_{\alpha_\beta}\}$ and a state $f \in \mathfrak{S}(\mathfrak{A})$ such that $f_{\alpha_\beta} \xrightarrow{w^*} f$. Therefore,
\begin{equation}\label{new2}
f_{\alpha_\beta}(a) \longrightarrow f(a) \quad \text{and} \quad f_{\alpha_\beta}(a^*a) \longrightarrow f(a^*a).
\end{equation}
From \eqref{new1} and \eqref{new2}, we infer that 
\[
|f(a)| = w_{\mathfrak{A}}(a) \quad \text{and} \quad f(a^*a) = \|a\|^2.
\]
By the GNS representation, there exists a Hilbert space $\mathscr{H}$, a representation $\pi : \mathfrak{A} \to \mathscr{B}(\mathscr{H})$, and a unit vector $\xi \in \mathscr{H}$ such that 
\[
|\langle \pi(a)\xi, \xi \rangle| = w_{\mathfrak{A}}(a) \quad \text{and} \quad \|\pi(a)\xi\|^2 = \|a\|^2.
\]
One can easily verify that $w(\pi(a)) = w_{\mathfrak{A}}(a)$ and $\|\pi(a)\| = \|a\|$. Since $\pi(a)$ is simultaneously norm-attaining and numerical radius attaining with respect to the same vector $\xi$, Theorem~\ref{op thm} ensures that $w(\pi(a)) = \|\pi(a)\|$, whence $w_{\mathfrak{A}}(a) = \|a\|$. This yields $|f(a)| = \|a\|$, and Lemma~\ref{para cor} implies $a \parallel \boldsymbol{1}$.
\end{proof}

Recall that an element $a \in \mathfrak{A}$ is called \emph{normaloid} if its spectral radius coincides with its norm, i.e., $r_{\mathfrak{A}}(a) = \|a\|$, where $r_{\mathfrak{A}}(a) = \sup\{|\omega| : \omega \in \sigma_{\mathfrak{A}}(a)\}$ and $\sigma_{\mathfrak{A}}(a)$ denotes the spectrum of $a$. It is well known that $a \in \mathfrak{A}$ is normaloid if and only if $w_{\mathfrak{A}}(a) = \|a\|$. Every normal, unitary, or self-adjoint element is obviously normaloid.

As an immediate consequence of Theorem~\ref{norm thm}, we obtain several characterizations of normaloid elements via the Davis--Wielandt radius.

\begin{corollary}\label{norm corol}
Let $a \in \mathfrak{A}$. Then the following assertions are equivalent:
\begin{enumerate}
    \item[(i)] $dw_{\mathfrak{A}}(a) = \sqrt{w_{\mathfrak{A}}^2(a) + \|a\|^4}$.
    \item[(ii)] $a$ is normaloid.
    \item[(iii)] $dw_{\mathfrak{A}}(a) = \|a\|\sqrt{1 + \|a\|^2}$.
    \item[(iv)] $a^*a \leq w_{\mathfrak{A}}^2(a)\boldsymbol{1}$.
\end{enumerate}
\end{corollary}

\begin{proof}
The equivalence (i) $\iff$ (ii) follows directly from the proof of Theorem~\ref{norm thm}.

\noindent (i) $\implies$ (iii): This implication follows immediately from the equivalence (i) $\iff$ (ii) by substituting $w_{\mathfrak{A}}(a) = \|a\|$.

\noindent (iii) $\implies$ (i): Let $dw_{\mathfrak{A}}(a) = \|a\|\sqrt{1 + \|a\|^2}$. Since $w_{\mathfrak{A}}(a) \leq \|a\|$, inequality \eqref{dwinq} gives
\[
\|a\|\sqrt{1 + \|a\|^2} = dw_{\mathfrak{A}}(a) \leq \sqrt{w_{\mathfrak{A}}^2(a) + \|a\|^4} \leq \|a\|\sqrt{1 + \|a\|^2}.
\]
Thus, the inequalities collapse to equalities, proving (i).

\noindent (i) $\iff$ (iv): From the equivalence (i) $\iff$ (ii), we know that $dw_{\mathfrak{A}}(a) = \sqrt{w_{\mathfrak{A}}^2(a) + \|a\|^4}$ if and only if $a$ is normaloid, which is equivalent to $w_{\mathfrak{A}}(a) = \|a\|$. This equality holds if and only if $\sup_{f \in \mathfrak{S}(\mathfrak{A})} f(a^*a) = \|a^*a\| = \|a\|^2 = w_{\mathfrak{A}}^2(a)$, implying that $f(a^*a) \leq w_{\mathfrak{A}}^2(a)$ for all $f \in \mathfrak{S}(\mathfrak{A})$. This is equivalent to stating that $f(w_{\mathfrak{A}}^2(a)\boldsymbol{1} - a^*a) \geq 0$ for every state $f$, which yields the desired positive operator inequality $a^*a \leq w_{\mathfrak{A}}^2(a)\boldsymbol{1}$.
\end{proof}

\begin{remark}
    Corollary \ref{norm corol} proves that the equality $dw_{\mathfrak{A}}(a) = \sqrt{w_{\mathfrak{A}}^2(a) + \|a\|^4}$ holds in Theorem \ref{dw-prop} if and only if $a$ is normaloid.
\end{remark}

\begin{corollary}
Let $a \in \mathfrak{A}$. Then the following conditions are equivalent:
\begin{enumerate}
    \item[(i)] $dw_{\mathfrak{A}}(a) = \sqrt{w_{\mathfrak{A}}^2(a) + \|a\|^4}$.
    \item[(ii)] There exists a state $f \in \mathfrak{S}(\mathfrak{A})$ such that $|f(a)| = \|a\|$.
\end{enumerate}
\end{corollary}

\begin{proof}
The assertion follows immediately from Lemma~\ref{para cor} and Theorem~\ref{norm thm}.
\end{proof}

We next examine the behavior of norm-parallelism when applied to powers and adjoints of an element parallel to the identity.

\begin{lemma}\label{norm lem}
Let $a \in \mathfrak{A}$. Then the following assertions hold:
\begin{enumerate}
    \item[(i)] $a \parallel \boldsymbol{1} \iff a^*a \parallel a^*$.
    \item[(ii)] $a \parallel \boldsymbol{1} \iff a \parallel a^*$.
\end{enumerate}
\end{lemma}

\begin{proof}
(i) Suppose $a \parallel \boldsymbol{1}$. Then $\|a + \lambda \boldsymbol{1}\| = \|a\| + 1$ for some $\lambda \in \mathbb{T}$. By \cite[Theorem~3.3.6]{murphy2014c}, there exists a state $f \in \mathfrak{S}(\mathfrak{A})$ such that
\[
f\big((a + \lambda \boldsymbol{1})^*(a + \lambda \boldsymbol{1})\big) = \|(a + \lambda \boldsymbol{1})^*(a + \lambda \boldsymbol{1})\| = \|a + \lambda \boldsymbol{1}\|^2 = (\|a\| + 1)^2.
\]
Expanding the left-hand side yields
\begin{align*}
    (\|a\| + 1)^2 &= f(a^*a) + \lambda f(a^*) + \bar{\lambda} f(a) + 1 \\
    &\leq \|a^*a\| + \|\lambda a^*\| + \|\bar{\lambda}a\| + 1 \\
    &= \|a\|^2 + 2\|a\| + 1 = (\|a\| + 1)^2.
\end{align*}
This forces the inequalities to be equalities, from which we get $f(\lambda a^*) = \|a\|$ and $f(a^*a) = \|a^*a\|$. Consequently,
\[
\|a^*a\| + \|a^*\| = f(a^*a) + f(\lambda a^*) = f(a^*a + \lambda a^*) \leq \|a^*a + \lambda a^*\| \leq \|a^*a\| + \|a^*\|.
\]
Thus, $\|a^*a + \lambda a^*\| = \|a^*a\| + \|a^*\|$ for $\lambda \in \mathbb{T}$, meaning $a^*a \parallel a^*$.

Conversely, assume that $a^*a \parallel a^*$, or equivalently, $\|a^*a + \lambda a^*\| = \|a^*a\| + \|a^*\| = \|a\|^2 + \|a\|$ for some $\lambda \in \mathbb{T}$. By Lemma~\ref{para cor}, there exists a state $f \in \mathfrak{S}(\mathfrak{A})$ such that $|f(a^*a + \lambda a^*)| = \|a^*a + \lambda a^*\| = \|a\|^2 + \|a\|$. Therefore, we obtain
\[
\|a\|^2 + \|a\| = |f(a^*a + \lambda a^*)| \leq |f(a^*a)| + |f(a^*)| \leq \|a\|^2 + \|a\|.
\]
This implies $|f(a^*)| = \|a\|$, so there exists $\lambda' \in \mathbb{T}$ such that $f(a^*) = \lambda' \|a\|$, i.e., $f(a) = \overline{\lambda'} \|a\|$. Thus,
\[
\|a\| + 1 = f(\overline{\lambda'}a + \boldsymbol{1}) \leq \|\overline{\lambda'}a + \boldsymbol{1}\| = \|a + \lambda'\boldsymbol{1}\| \leq \|a\| + 1,
\]
which shows $\|a + \lambda'\boldsymbol{1}\| = \|a\| + 1$ for $\lambda' \in \mathbb{T}$. Hence, $a \parallel \boldsymbol{1}$.

(ii) A completely analogous approach establishes $a \parallel \boldsymbol{1} \iff a \parallel a^*$.
\end{proof}

By synthesizing these tools, we easily obtain the following theorem.

\begin{theorem}
Let $a\in\mathfrak{A}$. Then the following statements are equivalent:
\begin{enumerate}
    \item [(i)] $dw_{\mathfrak{A}}(a)=\sqrt{w_{\mathfrak{A}}^2(a)+\|a\|^4}$.
    \item [(ii)] There exists a state $f\in \mathfrak{S}(\mathfrak{A})$ such that $|f(a^2)|=\|a\|^2$.
    \item[(iii)] There exists a state $f\in \mathfrak{S}(\mathfrak{A})$ such that $|f(aa^*a)|=\|a\|^3$.
\end{enumerate}
\end{theorem}
\begin{proof}
(i) $\iff$ (ii): By Theorem~\ref{norm thm} and Lemma~\ref{norm lem}, condition (i) is equivalent to $a \parallel a^*$. Applying the Lemma~\ref{para cor} to the pair $(a, a^*)$, we find that $a \parallel a^*$ if and only if there exist a state $g \in \mathfrak{S}(\mathfrak{A})$ and a scalar $\lambda \in \mathbb{T}$ such that $g(a^* a^*) = \lambda \|a\| \|a^*\| = \lambda \|a\|^2$. Noting that $g(a^* a^*) = g((a^2)^*)$, we take the complex conjugate to find $g(a^2) = \overline{\lambda} \|a\|^2$. Setting $f = g$ and taking the absolute value yields $|f(a^2)| = \|a\|^2$.

(i) $\iff$ (iii): Similarly, Theorem~\ref{norm thm} and Lemma~\ref{norm lem} establish that condition (i) holds if and only if $a^*a \parallel a^*$. Applying the Lemma~\ref{para cor} to the elements $a^*a$ and $a^*$, we deduce that $a^*a \parallel a^*$ if and only if there exists a state $g \in \mathfrak{S}(\mathfrak{A})$ and $\lambda \in \mathbb{T}$ such that $g((a^*a)^* a^*) = \lambda \|a^*a\| \|a^*\|$. Since $(a^*a)^* = a^*a$, this simplifies to $g(a^*a a^*) = \lambda \|a\|^3$. Taking the complex conjugate and using the state properties, we obtain a state $f \in \mathfrak{S}(\mathfrak{A})$ such that $|f(aa^*a)| = \|a\|^3$, completes the proof.
\end{proof}

Our final theorem provides a necessary and sufficient condition for the extreme scaling condition $dw_{\mathfrak{A}}(a) = \sqrt{2} w_{\mathfrak{A}}(a)$ under the assumption of parallelism to identity.

\begin{theorem}
Let $a \in \mathfrak{A}$ be such that $a \parallel \boldsymbol{1}$. Then the following conditions are equivalent:
\begin{enumerate}
    \item[(i)] $\|a\| = 1$.
    \item[(ii)] $dw_{\mathfrak{A}}(a) = \sqrt{2} w_{\mathfrak{A}}(a)$.
\end{enumerate}
\end{theorem}

\begin{proof}
Let $a \in \mathfrak{A}$ be such that $a \parallel \boldsymbol{1}$. 

(i) $\implies$ (ii): Let $\|a\| = 1$. Then by Theorem~\ref{norm thm} and Corollary~\ref{norm corol}, we have $w_{\mathfrak{A}}(a) = \|a\| = 1$. Thus,
\[
dw_{\mathfrak{A}}(a) = \sqrt{w_{\mathfrak{A}}^2(a) + \|a\|^4} = \sqrt{1 + 1} = \sqrt{2} = \sqrt{2} w_{\mathfrak{A}}(a).
\]

(ii) $\implies$ (i): Assume $dw_{\mathfrak{A}}(a) = \sqrt{2} w_{\mathfrak{A}}(a)$. Since $a \parallel \boldsymbol{1}$, Theorem~\ref{norm thm} guarantees $dw_{\mathfrak{A}}(a) = \sqrt{w_{\mathfrak{A}}^2(a) + \|a\|^4}$. Squaring both sides yields
\begin{align*}
    w_{\mathfrak{A}}^2(a) + \|a\|^4 = 2 w_{\mathfrak{A}}^2(a) &\implies \|a\|^4 = w_{\mathfrak{A}}^2(a) \\
    &\implies \|a\|^2 = w_{\mathfrak{A}}(a).
\end{align*}
By Corollary~\ref{norm corol}, $w_{\mathfrak{A}}(a) = \|a\|$ since $a \parallel \boldsymbol{1}$. Substituting this into the relation gives $\|a\|^2 = \|a\|$. Since $a$ is a non-zero element, we conclude that $\|a\| = 1$.
\end{proof}

\section{Conclusion}\label{sec5}

In conclusion, this paper establishes a rigorous state-space framework for the algebraic Davis--Wielandt shell and radius in unital $C^*$-algebras, successfully bypassing classical spatial non-convexity irregularities. By characterizing the shell's boundary geometry, deriving sharp metric inequalities for sums of elements, and clarifying the structural interplay with norm-parallelism, these results significantly advance numerical radius theory.

\section*{Declarations}
\begin{itemize}
    \item \textit{Author contributions:} The authors contributed equally to this work and approved the final manuscript.
    \item \textit{Funding:} This research did not receive external funding.
    \item \textit{Data availability:} No data were used to support this study.
    \item \textit{Conflict of interest:} The authors declare that they have no conflict of interest.
\end{itemize}

\end{document}